\documentclass[11pt,reqno]{amsart}
\usepackage{stmaryrd}
\usepackage{amsfonts,amsmath,amssymb}
\usepackage{mathrsfs}
\usepackage[latin1]{inputenc}
\newtheorem{theorem}{Theorem}[section]

\newtheorem{corollary}[theorem]{Corollary}

\newtheorem{definition}[theorem]{Definition}

\newtheorem{lemma}[theorem]{Lemma}

\newcommand{\R}{\mathbb{R}}

\newcommand{\N}{\mathbb{N}}

\newcommand{\M}{\mathrm{M}}

\newcommand{\g}{\mathrm{g}}
\newcommand{\dd}{\mathrm{d}}
\newcommand{\dv}{\,\mathrm{dv}^{n}}

\newcommand{\dvv}{\,\mathrm{dv}^{2n-1}}
\newcommand{\ds}{\,\mathrm{d\sigma}^{n-1}}

\newcommand{\dss}{\,\mathrm{d\sigma}^{2n-2}}
\newcommand{\I}{\mathcal{I}}

\newcommand{\s}{\mathbb{S}}

\newcommand{\D}{\mathcal{D}}
\newcommand{\X}{\mathcal{X}}
\newcommand{\p}{\partial}
\newcommand{\norm}[1]{\|#1\|}
\newcommand{\abs}[1]{|#1|}
\newcommand{\set}[1]{\left\{#1\right\}}
\newcommand{\para}[1]{(#1)}

\newcommand{\seq}[1]{\left<#1\right>}

\newcommand{\To}{\longrightarrow}
\newcommand{\vv}{\mathrm{v}}
\newcommand{\hh}{\mathrm{h}}
\usepackage{graphicx}
\newcommand{\dive}{\textrm{div}}
\usepackage{mathrsfs}
\usepackage[latin1]{inputenc}
\usepackage[T1]{fontenc}
\usepackage{color,xspace}
\usepackage[usenames,dvipsnames]{pstricks}
\allowdisplaybreaks
\usepackage{lmodern}
\usepackage{times}
\allowdisplaybreaks
\usepackage[pagewise]{lineno}
\usepackage{hyperref}
\hypersetup{
    colorlinks=true,
    linkcolor=magenta,%blue,
    filecolor=magenta,
    urlcolor=cyan,
    pdftitle={Overleaf Example},
    pdfpagemode=FullScreen,
    }
\begin{document}

\title[Linear Boltzmann equation]{An inverse problem for the time-dependent linear Boltzmann equation  in a Riemannian setting}
%\author[M. Bellassoued]{Mourad~Bellassoued}
\author[Z. Rezig]{Zouhour Rezig}
% in alphabetical order
%\address{University of Tunis El Manar, National Engineering School of Tunis, ENIT-LAMSIN, B.P. 37, 1002 Tunis, Tunisia}
%\email{mourad.bellassoued@enit.utm.tn}
\address{University of Tunis El Manar, Faculty of Sciences of Tunis, ENIT-LAMSIN, B.P. 37, 1002 Tunis, Tunisia}
\email{zouhour.rezig@fst.utm.tn}
%\date{\today}
%%%%%%%%%%%%%%%%%%%%%%%%%%%%%%%%%
%%%%%%%%%%%%%%%%%%%%%%%%%%%%%%%%

%%%%%%%%%%%%%%%%%%%%%%%%%%%%%%%%
\begin{abstract}
 The linear Boltzmann equation governs the absorption and scattering of a population of particles in a medium with an ambient field, represented by a Riemannian metric, where particles follow geodesics. In this paper, we study the possible issues of uniqueness and stability in recovering the absorption and scattering coefficients from the boundary knowledge of the albedo operator. The albedo operator takes the incoming
flux to the outgoing flux at the boundary. For simple compact  Riemannian manifolds of dimension $n \geq 2$, we study the stability of  the absorption coefficient from the albedo operator up to a gauge transformation. We derive that when the absorption coefficient is isotropic then the albedo operator determines uniquely the absorption coefficient and we establish a stability estimate. We also give an identification result  for the reconstruction of the scattering parameter. The approach in this work is based on the construction of  suitable  geometric optics solutions and  the use of the invertibility of the geodesic ray transform.
%\tableofcontents
\end{abstract}
\maketitle
%%%%%%%%%%%%%%%%%%%%%%%%%%%%%%%%%%%%%%%%%%%%
\section{Introduction and main results}
\subsection{Statement of the problem}
Let $\M\subset\R^n$, $n\geq 2$, be a convex compact domain with smooth boundary $\p\M$, equipped with  a Riemannian metric $\g$ of class $C^{\infty}$. In a local coordinates system, we write $\g=(\g_{jk})$ and by $(\g^{jk})$ we denote the inverse of $\g$ and  by $\det\,(\g_{jk})$ the determinant of $\g$. The inner product and the norm on the tangent space $T_x\M$ are respectively  denoted $\seq{,}$ and $|\cdot|$.  We define the sphere bundle of $\M$ by setting
\begin{align*}
S\M=\set{(x,\theta)\in T\M;\,\abs{\theta}=1}.
\end{align*}
We fix $T>0$ and we let $a$ and $k$ be two positive and bounded functions.
We are interested in the linear Boltzmann transport equation
\begin{equation}\label{1.1}
\partial_t u(t,x,\theta)+\D u(t,x,\theta)+a(x,\theta)u(t,x,\theta)=\I_k[u](t,x,\theta)\;\;   \textrm{in }\;(0,T)\times S\M,
\end{equation}
where $\D$ is the derivative along the geodesic flow and $\I_k$ is the integral operator with kernel  $k(x,\theta,\theta')$ defined by
\begin{equation*}\label{Ik}
\I_k[u](t,x,\theta)=\displaystyle \int_{S_x\M} k(x,\theta,\theta')u(t,x,\theta') \dd\omega_x( \theta'),
\end{equation*}
where $\dd\omega_x(\theta)$ is the volume form on the unit sphere $S_x\M=S\M \cap T_x\M$ at $x$. We can identify $S_x\M$ to $\mathbb{S}^{n-1}$ the unit sphere of $\R^n$.\\
The case of an Euclidean metric corresponds to transport in materials with a constant index of refraction.  For a general metric, the integro-differential transport equation $(\ref{1.1})$ governs the evolution of a population of particles  in a
medium with varying, anisotropic index of refraction like it is the case of the evolution of neutrons in a reactor \cite{[CZ]} \cite{[DM]}. A particle is described, at time $t$, by its position $x$ and the direction of its velocity $\theta$. The solution $u(t,x,\theta)$ represents the density of particles at the moment $t$, at $(x,\theta)$. The equation $(\ref{1.1})$ models only scattering of particles from the medium and assumes that particle-to-particle interaction may be neglected (which leads to discard nonlinear terms) and that all scattering occurrences preserve speed (i.e. elastic scattering; that is $k(x,\theta,\theta')=0$ for all $|\theta| \neq |\theta'|$).  In absence of interaction with the medium, the term $\D u$ describes the motion of a particle following a geodesic. If $\g$ is Euclidean
then $\D$ is the directional derivative: $\D u(x, \theta) = \theta \cdot \nabla_x u(x,\theta)$. The term $a(x,\theta)u(t,x,\theta)$ is the particle loss at the moment $t$ at $(x,\theta)$ due to absorption, quantified by the function $a(x,\theta)$. The scattering (or collision) coefficient $k(x,\theta,\theta')$ represents the probability density for the particles  at $x$ to change from speed $\theta'$ to $\theta$. So the right-hand side of $(\ref{1.1})$ represents the production, due to scattering, at $x\in \M$ of particles with velocity $\theta$ from particles with velocity $\theta'$.\\
Let $(x,\theta)\in S\M$, we denote by $\gamma_{x,\theta}(t)$ the geodesic satisfying $\gamma_{x,\theta}(0)=x$ and $\dot{\gamma}_{x,\theta}(0)=\theta$. Under assumptions on the geometry of $(\M,\g)$, discussed in Definition \ref{simple}, the geodesic $\gamma_{x,\theta}(t)$ is defined on a maximal finite segment $[\tau_-(x,\theta),\tau_+(x,\theta)]$. Thus, the geodesic flow is defined by
\begin{equation}\label{flowphit}
\phi_t:S\M\to S\M,\quad \phi_t(x,\theta)=(\gamma_{x,\theta}(t),\dot{\gamma}_{x,\theta}(t))
,\quad t\in [\tau_-(x,\theta),\tau_+(x,\tau)],
\end{equation}
and the derivative of a function $v$ along the geodesic flow, is defined as follows : for  $(x,\theta)\in S\M$,
\begin{equation}\label{D}
\D v(x,\theta)=\p_s(v(\phi_s(x,\theta)))_{|s=0}.
\end{equation}

In order to find a solution of $(\ref{1.1})$, we must know the incoming flow at the boundary. We first define  the incoming and outgoing bundles on $\p\M$ as
\begin{equation*}\label{incom outcom}
\p_{\pm}S\M =\set{(x,\theta)\in S\M,\, x \in \p \M,\, \pm\seq{\theta,\nu(x)}\geq 0},
\end{equation*}
where $\nu$ is the unit vector of the outer normal to the boundary. The boundary of the manifold $ S\M$, defined as $\p S\M=\set{(x,\theta)\in S\M,\, x \in \p \M}$, is the union of the two compact submanifolds $\p_+ S\M$ and $\p_-S\M$.\\
To simplify the presentation, we denote $$Q_T=(0,T)\times S\M,\;\;\Sigma=(0,T)\times \p S\M,\;\;\Sigma_\pm=(0,T)\times \p_\pm S\M.$$
Under some assumptions on the geometry of $(\M,\g)$ and on the functions $a$ and $k$, discussed below, the initial boundary value problem
 \begin{equation}\label{IBVP}
\left\{
\begin{array}{lll}
\big(\partial_t u+\D u+a(x,\theta)\big)u(t,x,\theta)=\I_k[u](t,x,\theta)  & \textrm{in }\; Q_T,\cr
u(0, \cdot,\cdot )=0 & \textrm{in }\; S\M ,\cr
u=f_- & \textrm{on } \; \Sigma_-,
\end{array}
\right.
\end{equation} has a unique solution.\\
%%%%%%%%%%%%%%%%%%%%%%%%%%%%
 The medium is probed with the given radiation $ f_- $ and the exiting radiation detected on $\p _+ S\M$ measures the response to this excitation. In this paper, we deal with the inverse problem of determining the absorption coefficient $a$ and the collision kernel $k$ appearing in (\ref{IBVP}).  The data  are given by the incoming and the outgoing flows $u_{|\Sigma_{\pm}}$. We can define the albedo operator  \begin{equation*}\label{Albedo}
\mathcal{A}_{a,k}:  f_- \longmapsto u_{|\Sigma_+},
\end{equation*}
 where $u$ is the solution of  $(\ref{IBVP})$.
In a first step, we study the uniqueness and the stability issues in the determination of the absorption coefficient $a$ from the knowledge of the albedo operator. In fact, when the coefficient $a$ is anisotropic; that is depending on both $x$ and $\theta$, it is possible to have media of differing attenuation and scattering properties which yield the same albedo operator.
The non-uniqueness is characterized by the invariance of the albedo operator under a gauge transformation. Indeed, for a positive function $l \in L^{\infty}(S\M)$  with $\D l \in L^{\infty}(S\M)$ and such that $l=1$ on $\p S\M$, if we set
\begin{equation}\label{gauge}
\tilde{a}(x,\theta)= a(x,\theta) -\D \log l(x, \theta),\;\tilde{k}(x, \theta,\theta') = k(x, \theta,\theta')\frac{l(x,\theta)}{l(x,\theta')},
\end{equation}
then $u$ satisfies (\ref{IBVP}) if and only if $\tilde{u} = lu$ solves
\begin{equation*}
\partial_t \tilde{u}+\D \tilde{u}+a(x,\theta)\tilde{u}(t,x,\theta)=\I_k[\tilde{u}](t,x,\theta)  \; \textrm{in }\; Q_T.
\end{equation*}
Since $l=1$ on $\p S\M$ then $ u =\tilde{u}$ on $\p S\M$, so $\mathcal{A}_{a,k}=\mathcal{A}_{\tilde{a},\tilde{k}}$ (we refer to \cite{[CS2]} or \cite{[ST]} for more details). To remove this lack of uniqueness, we assume that the coefficient $a$ is isotropic; that is  it depends only on the position $x$, and we prove  that we can stably recover the  geodesic ray transform of $a$, namely the integrals of $a$ along geodesics of $g$, from $\mathcal{A}_{a,k}$. Then in order to reconstruct $a$, we add a geometric restriction on the metric $\g$ and we use a stability estimate inverting the  geodesic ray transform. The second step of this work consists in proving a uniqueness result for recovering the spatial part of the scattering coefficient $k$  from the albedo operator.\\
 Transport equation $(\ref{1.1})$ is used to model the absorption and scattering of near-infrared light; the determination of the  absorption and scattering properties of a medium from the measurement of the response to such transmitted light is known as optical tomography \cite{[A]} and  has been applied to the problems of medical imaging  \cite{[NW]} and optical molecular imaging \cite{[CBGA]}. We also find applications in remote sensing in the atmosphere \cite{[MD]} \cite{[S]}, in astrophysics to study the interstellar media \cite{[C]} and for light propagating in random media such a biological tissue \cite{[A]}\cite{[AS]}.
    \subsection{Relation to the litterature}
  In the Riemannian setting, an inverse problem for the stationary transport equation is treated in  \cite{[S]},  where $k(x,\theta,\theta')=k(x,\langle \theta,\theta'\rangle)$ is assumed to depend on the angle between $\theta$ and $\theta'$ and the problem is the recover of an isotropic source term. In \cite{[M]}, McDowall considered the stationary transport equation in simple manifolds with $a(x,\theta)=a(x,|\theta|)$ depending on the speed $|\theta|$ and not on the direction. It is shown that the albedo operator determines the absorption coefficient (in all dimensions) and the scattering coefficient $k(x,\theta,\theta')$ in dimension $3$ or greater. A time-dependent transport equation is considered in \cite{[F]}, with spatially varying refractive index $a(x)$ and with scattering coefficient $k(\theta,\theta')$ independent of position.\\
   More is known when the metric is assumed to be Euclidean. The uniqueness of the reconstruction of the absorption and scattering coefficients both in the time-dependent and time-independent settings was proved by Choulli and Stefanov in \cite{[CS1]}, \cite{[CS2]}. They studied the cases where $a$ depends on position and direction or only on position and where $k$ does not depend on $t$.  The authors used the scattering theory and its relationship with the albedo operator.  They proved that the albedo operator determines uniquely $a(x)$ and $k(x,\theta,\theta')$ when $a$ depends only on the position $x$ and that there is no uniqueness in the recover of $a(x,\theta)$ if it depends on both the position and the direction. They also gave a stability result for the reconstruction of $a(x)$. For two-dimensional domains in \cite{[T]}, a homogenous scattering kernel $k$ that is a function of two independent directions was shown to be uniquely determined. In \cite{[SU]}, the assumption of homogeneity is dropped and $k$ is assumed to be small relative to $a$ and the authors also gave a stability result. Stability in the time-independent case has also been analyzed in dimension $d=2, 3$ under smallness assumptions for the absorption and scattering parameters by Romanov \cite{[R1]}, \cite{[R2]}. In dimension d = 3, Wang \cite{[W]} obtained partial stability  results for recovering the coefficients  without smallness assumptions.  In three or higher  dimensions, Bal and Jollivet proved a stable recovery of the absorption and collision parameters in the stationary setting  \cite{[BJ1]} and in the time-dependent case \cite{[BJ2]}. Their approach is based on the singular decomposition of the Schwartz kernel of the albedo operator.  For time-depend Boltzmann equation where $a$ depends only on $x$, we can also refer to the work of  Cipolatti, Motta and Roberty \cite{[CMR]} and Cipolatti \cite{[C]} where they used geometric optics solutions. For the case where $a(t,x)$ depends on the time $t$ and the position $x$, we can consult the paper of Bellassoued and Boughanja \cite{[BB1]} where they established a log-type stability in the determination of $a(t,x)$ in different regions of the cylindrical domain by considering  different sets of data. The authors also proved that there is an obstruction to uniqueness appearing in a cloaking region of the domain and by enlarging the set of data, they showed that $a(t,x)$ can be stably recovered in larger subsets of the domain, including the whole domain. The same authors in another work \cite{[BB2]} studied the problem with $a(x,\theta)$ depending on $x $ and $\theta$ and $k(t,x,\theta,\theta')$ depending on $t$. They proved a stability result in recovering $a$ up to a gauge transformation from the albedo operator and that the scattering time-dependent coefficient can be uniquely determined in a precise subset of the domain.
%%%%%%%%%%%%%%%%%%%%%%%
\smallskip
%%%%%%%%%%%%%%%%%%
\subsection{The direct problem and the albedo operator}
%%%%%%%%%%%%%%%%%%%%%%%
We introduce some definitions and notations adopted in this paper. We use the Einstein notation; that is $a_ix^i$ means that we sum the terms $a_ix^i$ for all the indices $i$.  Let $(x^1,\dots, x^n)$ be a local coordinates system in  $ \M$. We denote by $\frac{\p}{\p x^i}$ the coordinates vector fields and by $dx^i$ the coordinates covector fields.
Let $(\theta^1,\dots,\theta^n)$ be the coordinates of a vector $\theta \in T_x\M$; that is $\theta=\theta^i \frac{\p}{\p x^i}$. Then the family of functions $(x^1, \dots,x^n,\theta^1,\dots,\theta^n)$ is a local coordinate system in $T\M$ \textit{associated} with $(x^1,\dots,x^n).$
In the sequel, we will only use coordinates systems on $T\M$ associated with a given system $(x^1, \dots,x^n)$.\\

For $x\in \M$, the Riemannian scalar product on $T_x\M$ induces the volume form on $S_x\M$,
denoted by $\dd \omega_x(\theta)$. In local coordinates, it is given by
$$\dd \omega_x(\theta)=\sqrt{\det\g} \, \sum_{k=1}^n(-1)^k\theta^k \dd \theta^1\wedge\cdots\wedge \widehat{\dd \theta^k}\wedge\cdots\wedge \dd \theta^n.$$
As usual, the notation $\, \widehat{\cdot} \,$ means that the corresponding factor has been dropped.
We introduce the volume form $\dvv$ on the manifold $S\M$ by
$$
\dvv (x,\theta)=\dd\omega_x(\theta)\wedge \dv,
$$
where $\dv$ is the Riemannnian volume form on $\M$. By Liouville's theorem, the form $\dvv$ is preserved by the geodesic flow. The
corresponding volume form on the boundary $\p S\M =\set{(x,\theta)\in S\M,\, x\in\p \M}$ is given
by
$$
\dss=\dd\omega_x(\theta) \wedge \ds,
$$
where $\ds$ is the volume form of $\p \M$.\\
%%%%%%%%%%%%%%%%%%%%%%%%%%%%%%%%%%%%%%%%%%%%%%%%%%%%%%
Take an integer $p \geq 1$, and set $$L^p(S\M)=L^p(S\M,\dvv)\;\mbox{  and  } L^p(S_x\M)=L^p(S_x\M,\dd \omega_x(\theta)).$$ Let us introduce the admissible sets for the coefficients $a$ and $k$. We fix $C_0,C_1,C_2>0$
% and $\eta > 1+\frac{n}{2}$
and we set
\begin{equation}\label{A_0}
  \mathscr{A}_0(C_0)=\set{a \in C^{1}(S\M),\;\norm{a}_{L^{\infty}(S\M)}\leq C_0 }.
\end{equation}
For the scattering coefficient, we denote by $\mathfrak{K}$ the set of  functions $k$ satisfying
  \begin{equation}\label{K}
\left\{
\begin{array}{lll}
k \geq 0, \cr
\forall x \in \M, \, k(x,\cdot,\cdot)\in L^2(S_x\M\times S_x\M)\cr
x \longmapsto \| k(x,\cdot,\cdot) \|_{L^2(S_x\M\times S_x\M)} \in L^{\infty}(\M),
\end{array}
\right.
\end{equation} and we consider the admissible set
\begin{equation}\label{k}
  \mathscr{K}(C_1,C_2)=  \set{k\in \mathfrak{K},\;\; \int_{S_x\M}
k(x,\theta,\theta') \dd\omega_x(\theta')
 \leq C_1, \;
    \int_{S_x\M} k(x,\theta,\theta') \dd\omega_x(\theta) \leq C_2}.
\end{equation}
%%%%%%%%%%%%%%%%%%%%%%%%%%%%%%%%%%%%%%%%%%%%%%%%%%%%%%%%%%%%%%%%%%%%%%%%%%%%%%%%%%%%
\subsubsection{Admissible manifolds}
%%%%%%%%%%%%%%%%%%%%%%%%%%%%%%%%%%%%%%%%%%%%%%%%%%%%%%%%%%%%%%%%%%%%%%%%%%%%%%%%%%%%%
To investigate the inverse problem associated to the boundary value problem (\ref{IBVP}), we must specify the geometry in which we will work.
 For $\theta \neq 0$, we define the time-to-boundary functions $\tau_{\pm}:S\M \longrightarrow \R$ by $$\tau_{+}(x,\theta)=\min \set{t>0,\; \gamma_{x,\theta}(t)\in \p\M},\quad \tau_{-}(x,\theta)=-\min\set{t>0,\; \gamma_{x,\theta}(-t)\in \p\M} .$$ The geodesic flow $\phi_t$ defined in $(\ref{flowphit})$ satisfies $\phi_t\circ\phi_s=\phi_{t+s}$ and we have the following properties
$$
\tau_-(x,\theta)=0,\quad (x,\theta)\in\p_-S\M,\quad \tau_+(x,\theta)=0,\quad (x,\theta)\in\p_+S\M,
$$
$$
\tau_-(\phi_t(x,\theta))=\tau_-(x,\theta)-t,\quad \tau_+(\phi_t(x,\theta))=\tau_+(x,\theta)-t.
$$ In particular, for $(x,\theta)\in\p_- S\M$, $\gamma_{x,\theta} : [0,\tau_+(x,\theta)] \to \M$ is the maximal
geodesic satisfying the initial conditions $\gamma_{x,\theta}(0) = x$ and $\dot{\gamma}_{x,\theta}(0) = \theta$.
%Let $\theta \neq 0$.
The life time of an orbit $\gamma_{x,\theta} $ in $\M$ is defined by $\tau(x,\theta)=\tau_+(x,\theta)-\tau_-(x,\theta)$.\\
 If $\p \M$ is $C^1$ then $ \forall (x,\theta) \in S \M,\;\tau(x,\theta) > 0$. If $\M$ is bounded then $x \mapsto \tau(x,\theta)$ is bounded on $\M$. If $\M$ is convex then $x \mapsto \tau(x,\theta)$ is $C^1$ on the set $\set{x\in \M, \tau(x,\theta) < \infty}$.
 In general, the life time of an orbit $\tau(x,\theta)$ might be infinite. In this paragraph, we put restrictions on the metric of $\M$ that ensure that $\tau_{\pm}$ are well-defined and finite. Concerning smoothness properties of  $\tau_{\pm}(x,\theta)$, we can see that these functions are smooth near a point $(x,\theta)$ such that the geodesic $\gamma_{x,\theta}(t)$ intersects $\p\M$ transversely for $t=\tau_{\pm}(x,\theta)$.
For a Riemannian manifold $(\M,\g)$ with boundary $\p\M$, we denote by $\nabla$ the Levi-Civita connection on $(\M,\g)$. For a point $x \in \p\M$, we define the second quadratic form of the boundary on the space $T_x(\p\M)$ by $\Pi(\xi,\xi)=\seq{\nabla_\xi\nu,\xi}$, $\xi\in T_x(\p\M)$. We say that the boundary is strictly convex if the form is positive-definite for all $x \in \p\M$ (see \cite{[Sh1]}).
By strict convexity of $\p\M$, the functions $\tau_{\pm}(x,\xi)$ are smooth on $T\M \setminus T(\p\M).$ In fact, all points of $T\M \cap T(\p\M)$ are singular for $\tau_{\pm}$; since some derivatives of these functions are unbounded in a neighbourhood of such points. In particular, $\tau_{+}$ is smooth on $\p_-SM$, see Lemma 4.1.1 of \cite{[Sh1]}.
%\begin{definition}\label{nontrapping}
\smallskip
We say that a compact Riemannian manifold $(\M,\g)$  with boundary is a convex non-trapping manifold, if the boundary $\p \M$ is strictly convex,
and all geodesics have finite length in $\M$. An important subclass of convex non-trapping manifolds are simple manifolds.
%%%%%%%%%
\begin{definition}(Simple manifolds)\label{simple}
We say that a compact Riemannian manifold $(\M,\g)$ with boundary (or that the metric $\g$) is \textit{simple}, if $\p\M$ is strictly convex with respect to $\g$, and for any $x\in \M$, the exponential map $\exp_x:\exp_x^{-1}(\M)\to \M$ is a (global) diffeomorphism. The latter means that every two points $x$, $y \in \M$ are joined by a unique geodesic smoothly depending on $x$ and $y$.
\end{definition}
Any simple $n$-dimensional Riemannian manifold is diffeomorphic to a closed ball in $\R^n$, and that all geodesics in $\M$ have finite length. Another important property is that such manifold  can be extended to a simple manifold $\widetilde{\M}$ such that $\widetilde{\M}^{\textrm{int}}\supset\M$, we refer to \cite{[McST]} for more details.
%%%%%%%%%%%%%%%%%%%%%%%%%%%%%%%%%%%%%%%%%%%%%%%%%%%%%%%%%%
\subsubsection{Well-posedness of the linear Boltzmann equation}
%%%%%%%%%%%%%%%%%%%%%%%%%%%%
%In order to pose the boundary value problem, we must specify the volume form on $\p_{\pm} S\M$.
To well define the albedo operator, we consider the spaces $L^p_\mu(\p_{\pm }S\M)$ of $p$-integrable functions with respect to the measure
 \begin{equation}\label{mesure image}
 \mu(x,\theta)\dss \mbox{   with }\mu(x,\theta)=\abs{\seq{\theta,\nu(x)}},
   \end{equation} (see  \cite{[McST]}) and the spaces $\mathcal{L}^p(\Sigma_\pm)=L^p(0,T;L^p_\mu(\p_+S\M))$ endowed with the norm $$\|f\|_{\mathcal{L}^p(\Sigma_\pm)}=
\|f\|_{L^p(0,T;L^p_\mu(\p_+S\M))}.$$
Let $(x,\theta)\in S\M$, define $(x',\theta')\in \p_- S\M$ such that \begin{equation}\label{x',theta'}
 (x',\theta')=\phi_{\tau_-(x,\theta)}(x,\theta).
    \end{equation}
We have the following usefull lemma.
\begin{lemma}\label{lemma0}\textit{Santalo formula.}
  Let $u\in L^1(S\M)$, then the equality
  $$\int_{\M}\int_{S_x\M} u(x,\theta) \dd \omega_x(\theta)\dv=\mp \int_{\p_{\pm}S\M} \int_{0}^{\tau_{\mp}(x',\theta')} u(\gamma_{x',\theta'}(t),\dot{\gamma}_{x',\theta'}(t)) dt \mu(x',\theta')\dss,$$ holds true.
\end{lemma}
% The Hilbert space $L^2_\mu(\p_+S\M)$ is endowed with the scalar product
%\-begin{equation}\label{2.4}
%\para{u,v}_\mu=\int_{\p_+S\M}u(x,\xi) \overline{v}(x,\xi) %\mu(x,\xi)\dss. \end{equation}
Define the functions space $\mathcal{W}^p$ in which the forward problem is well-posed by setting
$$\mathcal{W}^p=\set{u\in L^p(S\M),\; \D  u \in L^p(S\M);\;\int_{\p S\M}\abs{u(x,\theta)}^p \mu(x,\theta)\dss < \infty }.$$
$\mathcal{W}^p$ is a Banach space with respect to the norm $$\norm{u}_{\mathcal{W}^p}=\norm{u}_{L^p(S\M)}+\norm{\D u}_{L^p(S\M)}.$$ We assume that $a$ and $k$ are in the admissible sets, then the operators $u\mapsto au$ and $u\mapsto \I_k[u]$ are bounded on $L^p(S\M)$.
Consider the unbounded operator $$Tu=\D u+a u-\I_k[u]$$ with domain $$\mathrm{D}(T)=\set{u\in \mathcal{W}^p
;\;u_{|\Sigma_+}=0
}.$$
The following Lemma establishes the unique existence of a weak solution to the problem (\ref{IBVP}). Its proof is based on \cite{[DL]} (Theorem $3$ p.229) and completed in Lemma \ref{lemma}.
\begin{lemma}\label{lemma1}
% Assume that $\M$ is a simple compact Riemannian manifold.
   Let $T>0$, $p\geq 1$, $a\in \mathscr{A}_0(C_0)$ and $k\in \mathscr{K}(C_1,C_2)$. Assume that $f_-\in \mathcal{L}^p(\Sigma_-)$. Then the  problem (\ref{IBVP}) admits a unique weak solution $$u \in C^1([0,T];L^p(S\M))\cap  C^0([0,T];\mathcal{W}^p).$$
 Moreover, there exist a constant $C>0$ such that
\begin{equation*}\label{estimate Albedo Lp}
   \norm{u}_{\mathcal{L}^p(\Sigma_+)} \leq  C \norm{f_-}_{\mathcal{L}^p(\Sigma_-)}.
\end{equation*}
\end{lemma}
Thereby, the albedo operator
 \begin{equation}\label{Albedo Lp}
\begin{array}{lll}
\mathcal{A}_{a,k}:& \mathcal{L}^p(\Sigma_-)& \longrightarrow \mathcal{L}^p(\Sigma_+),
\cr &f_- &\longmapsto u_{|\Sigma_+},
\end{array}
 \end{equation} is well-defined,  linear and bounded, we denote by $\norm{\mathcal{A}_{a,k}}_p$ its norm.\\
In the study of our inverse problem, we will also deal with  the following initial boundary value problem for the linear Boltzmann equation with source term $q\in L^p(Q)$.
 \begin{equation}\label{IBVPsource}
\left\{
\begin{array}{lll}
\partial_t u(t,x,\theta)+\D u(t,x,\theta)+a(x,\theta)u(t,x,\theta)=\I_k[u](t,x,\theta)+q(t,x,\theta)  & \textrm{in }\; Q_T,\cr
u(0, \cdot,\cdot )=0 & \textrm{in }\; S\M ,\cr
u=0 & \textrm{on } \; \Sigma_-.
\end{array}
\right.
\end{equation}
\begin{lemma}\label{lemma2}
% Assume that $\M$ is a simple compact Riemannian manifold.
  Let $T>0$, $p\geq 1$, $a\in \mathscr{A}_0(C_0)$ and $k\in \mathscr{K}$. If $q\in L^p(Q_T)$ then the  problem (\ref{IBVPsource}) admits a unique weak solution $$u \in C^1([0,T];L^p(S\M))\cap  C^0([0,T];\mathrm{D}(T)).$$
 Moreover, there exist a constant $C>0$ such that
\begin{equation}\label{estimate source}
  \norm{u(t,\cdot,\cdot)}_{L^p(S\M)}+ \norm{u}_{\mathcal{L}^p(\Sigma_+)} \leq  C \norm{q}_{L^p(Q_T)},\;\forall t\in [0,T].
\end{equation}
\end{lemma}
\subsection{Main results}
%The main results of this paper are as follows.
%Let $(\M,\g)$ be a simple compact Riemannian manifold  of dimension $n \geq 2$, with $C^{\infty}$-smooth boundary $\p \M$ and assume that $\g$ is the identity near the boundary $\p \M$. We slightly extend $\M$ to a simple manifold $\widetilde{\M}^{\textrm{int}}\supset\M$ and we extend $\g$ to $\widetilde{\M}$ by the identity.
\begin{theorem}\label{th0}
 Let $T>\textrm{Diam}_\g(\M)$. There exist $C=C(\M,T,C_0)>0$ such that for any $k_1,k_2\in\mathscr{K}(C_1,C_2)$ and
$a_1,a_2\in\mathscr{A}_0(C_0)$ satisfying  \begin{equation}\label{cond th0}
              \p^\alpha_{x}a_1(x,\theta)= \p^\alpha_{x}a_2(x,\theta), \;(x,\theta) \in \p S\M, \alpha\in \N^n,\abs{\alpha} \leq 1,
 \end{equation} we have
\begin{equation}\label{ineq -T T}
\Big| \displaystyle \int_{-T}^{T}\, a(\gamma_{x,\theta}(s),\dot{\gamma}_{x,\theta}(s))\, ds  \Big|\leq C\norm{\mathcal{A}_{a_1,k_1}-\mathcal{A}_{a_2,k_2}}_1;\;a=a_1-a_2,
\end{equation} for any $(x,\theta)\in S\M$.
\end{theorem}
From this Theorem  we derive that there is an objection to the unique determination of $a$ when this last is anisotropic; that is it depends on both $x$ and $\theta$. More precisely, we prove that an anisotropic absorption coefficient $a$ can be uniquely determined from the albedo operator up to a gauge transformation.
\begin{corollary}\label{corol0}
If we   assume that $\mathcal{A}_{a_1,k_1}=\mathcal{A}_{a_2,k_2}$, then there exist a positive function $l(x,\theta)\in L^{\infty}(S\M)$ with $\D l \in L^{\infty}(S\M)$ and $l(x,\theta)=1$ for $(x,\theta)\in \p S\M$, such that we have   $$a_1(x,\theta)=a_2(x,\theta)+\D (\log l(x,\theta)),\;\forall (x,\theta) \in S\M.$$ \end{corollary}
To remove this lack of uniqueness, we assume that the absorption coefficient $a$ is isotropic; that is it does not depend on the velocity variable, i.e. $a(x,\theta)=a(x)$. Then let  $\eta > 1+\frac{n}{2}$ and set
\begin{equation}\label{A}
  \mathscr{A}(C_0,\eta)=\set{a \in W^{2,\infty}(\M),\;\norm{a}_{W^{2,\infty}(\M)}+ \norm{a}_{H^{\eta}(\M)}\leq C_0 }.
\end{equation}
Let us introduce some notations. Given $x\in\M$ and a 2-plane $\pi\subset T_x\M$, denote by $K(x,\pi)$ the sectional curvature of $\pi$ at $x$. For $\theta\in S_x\M$, put
$$
K(x,\theta)=\sup_{\pi;\,\theta \in\pi}\,K(x,\pi),\quad K^+(x,\theta)=\max\{0,K(x,\theta)\}.
$$
Define the following characteristic:
$$
k^+(\M,\g)=\sup_\gamma\int_{\tau_1}^{\tau_2}tK^+(\gamma(t),\dot{\gamma}(t))dt,
$$
where $\gamma\,:\,[\tau_1,\tau_2]\to\M$ ranges in the set of all unit speed geodesics in $\M$. We have the following stability result for the absorption coefficient $a$ when it depends only on the position.
\begin{theorem}\label{th1}
Let $(\M,\g)$ be a simple compact Riemannian manifold  of dimension $n \geq 2$, with $C^{\infty}$-smooth boundary $\p \M$, such that $k^+(\M,\g)<1$. Assume that $\g$ is the identity near the boundary $\p \M$ and let $T>\textrm{Diam}_\g(\M)$.
There exist  $C=C(\M,T,C_0,C_1,C_2)>0$ such that for any $k_1,k_2\in\mathscr{K}(C_1,C_2)$ and
$a_1,a_2\in\mathscr{A}_0(C_0,\eta)$ satisfying  \begin{equation}\label{cond th1}
              \p^\alpha_{x}a_1(x)= \p^\alpha_{x}a_2(x), \;x \in  \p \M, \;  \alpha\in \N^n,\abs{\alpha} \leq 2,
 \end{equation} the following estimate
\begin{equation*}\label{estimate1 th1}
\norm{a_1-a_2}_{L^2(\M)}\leq C\norm{\mathcal{A}_{a_1,k_1}-\mathcal{A}_{a_2,k_2}}_1^{\frac{1}{2}},
\end{equation*}
holds true.
Moreover, there exist $\delta\in(0,1)$ such that
\begin{equation*}\label{estimate2 th1}
\norm{a_1-a_2}_{C^0(\M)}\leq C\norm{\mathcal{A}_{a_1,k_1}-\mathcal{A}_{a_2,k_2}}_1^\frac{\delta}{2}.
\end{equation*}
Therefore
\begin{equation}\label{corol}
 \mbox{If }\mathcal{A}_{a_1,k_1}=\mathcal{A}_{a_2,k_2} \;\mbox{ then } a_1=a_2 \, \mbox{ in } \M.
\end{equation}
\end{theorem}
Now we no longer suppose that $a$ is isotropic and we assume that the scattering coefficient is of the form $k(x,\theta,\theta')=\sigma(x)\chi(\theta,\theta')$, where $\chi(\theta,\theta')$ is supposed known. We will need some regularity conditions on $k$ to  recover its spatial part $\sigma$.
\begin{theorem}\label{th2}
Let $(\M,\g)$ be a simple compact Riemannian manifold  of dimension $n \geq 2$, with $C^{\infty}$-smooth boundary $\p \M$, such that $k^+(\M,\g)<1$. Assume that $\g$ is the identity near the boundary $\p \M$. Let $T>\textrm{Diam}_\g(\M)$, $a_1,a_2\in\mathscr{A}_0(C_0)$ and $k_j\in \mathscr{K}(C_1,C_2)$, $j\in{1,2},$ such that $k_j(x,\theta,\theta')=\sigma_j(x)\chi(\theta,\theta')$ with $\sigma_j\in  C^1(\M)$, $\chi\in C(S_x(\M) \times S_x(\M)),\;\forall x \in \M$ and such that $\chi(\theta,\theta) \neq 0,\; \forall \theta \in S_x\M$. We also suppose that we have
\begin{equation*}
\p^\alpha_{x}a_1(x,\theta)= \p^\alpha_{x}a_2(x,\theta),\;\;\p^\alpha_{x}\sigma_1(x)= \p^\alpha_{x}\sigma_2(x), \;x\in \p \M, \alpha\in \N^n,\abs{\alpha} \leq 1.
\end{equation*}
We claim that if $\mathcal{A}_{a_2,k_2}=\mathcal{A}_{a_1,k_1}$ then $\sigma_1=\sigma_2$ on $\M$.\\
In particular, if $a_1$ and $a_2$ are isotropic then
$\mathcal{A}_{a_1,k_1}=\mathcal{A}_{a_2,k_2}$ implies that $(a_1,k_1)=(a_2,k_2)$.
 \end{theorem}
 The main results of this section are obviously available for simple compact Riemannian  manifolds provided with metrics close to the identity or with negatively curved metrics. The condition $k^+(\M,\g)<1$ prohibits to the metric having conjugate points and it will be necessary to invert the geodesic ray transform.\\
 This paper is organized as follows: In section $2$ we construct special solutions to (\ref{IBVP}) called geometric optics solutions. In section $3$, we use these solutions to prove the Theorem \ref{th0}, Corollary \ref{corol0} and Theorem \ref{th1}. The section $4$ is devoted to the proof of Theorem \ref{th2}.
 %%%%%%%%%%%%%%%%%%%%%%%%%%%%%%%%%%%%%%%%%%%%%%%%%%%%%
 %%%%%%%%%%%%%%%%%%%%%%%%%%%%%%%%%%%%%%%%%%%%%%%%%%%%%%%%%%
\section{Construction of geometric optics solutions}
\setcounter{equation}{0}
In this section, we use the WKB expansion method to built a family of solutions for the linear Boltzmann equation which will be used to prove the main results of this paper. We aim to construct geometric optics solutions $u_j,\;j=1,2$, for the two linear Boltzmann equations:
\begin{equation*}\label{eq1}
  \partial_t u_1(t,x,\theta)+\D u_1(t,x,\theta)+a(x,\theta)u_1(t,x,\theta)=\I_k[u_1](t,x,\theta)\;\; \textrm{in }\; Q_T,
\end{equation*}
and
\begin{equation*}\label{eq2}
  \partial_t u_2(t,x,\theta)+\D u_2(t,x,\theta)-a(x,\theta)u_2(t,x,\theta)=-\I^*_k[u_2](t,x,\theta)\;\; \textrm{in }\; Q_T,
\end{equation*}
of the form
\begin{equation*}\label{u1}
  u_1(t,x,\theta)=(\alpha_1 \beta_a)(t,x,\theta) e^{i\lambda(t+\tau_-(x,\theta))}+r_{1,\lambda} (t,x,\theta)
  \end{equation*}
and \begin{equation*}\label{u2}
  u_2(t,x,\theta)=(\alpha_2 \beta_{-a})(t,x,\theta) e^{-i\lambda(t+\tau_-(x,\theta))}+r_{2,\lambda} (t,x,\theta),
\end{equation*}
where $\lambda > 0$ is an arbitrary large parameter,  $\norm{r_{j,\lambda}}\rightarrow 0$ as $\lambda \rightarrow \infty$ for a suitable norm and
\begin{equation*}\label{I*k}
\I^*_k[u](t,x,\theta)=\displaystyle \int_{S_x\M} k(x,\theta,\theta')u(t,x,\theta) \dd\omega_x( \theta).
\end{equation*}
Let us begin the construction of such solutions. Fix $\epsilon > 0$ small enough to verify
\begin{equation}\label{condition T}
T > Diam_{\g}(\M)+2 \epsilon, \end{equation} and such that we can extend $\M$ to a simple compact Riemannian manifold $\widetilde{\M}\Supset \M$ satisfying the condition
\begin{equation}\label{M tilde}
  \forall x \in \widetilde{\M}\setminus \M,\;\;dist_{\g}(x,\M) < \epsilon.
  \end{equation} We also extend the metric $\g$ by the identity outside $\M$ and we keep the same notation $\g$ for this extension.  We take $a\in \mathscr{A}_0(C_0)$ that we extend to a $C^1$ function on $\R^n \times \mathbb{S}^{n-1}$ vanishing outside $\M$; that is $a(x,\theta)=0$ for all $x\in \R^n \setminus \M,\; \theta \in \mathbb{S}^{n-1}$.
Let \begin{equation}\label{phi} \phi\in C_0^{\infty}(\widetilde{\M}\setminus \M,C(\mathbb{S}^{n-1}))
\end{equation} that we also extend to $\R^n$ by zero: we mean that $\phi(x,\theta)=0,$ for $x\in \R^n \setminus \widetilde{\M}$ and $\theta \in \mathbb{S}^{n-1}$ and for $(x,\theta)\in S\M$. Henceforth, we keep the same notation for a function defined on $\M$ and its extension to $\R^n$.
 \begin{lemma}\label{lemma apha}
Let $\phi$ satisfying (\ref{phi}) and consider
\begin{equation*}\label{alpha}
\alpha(t,x,\theta)=\phi(\gamma_{x,\theta}(-t),\dot{\gamma}_
{x,\theta}(-t)),\;\; \forall(t,x,\theta)\in \R \times S\M.
\end{equation*}
Then $\alpha$ satisfies the transport equation \begin{equation*}\label{eq alpha}
\p _t\alpha(t,x,\theta)+\D \alpha(t,x,\theta)=0, \;\; \forall(t,x,\theta)\in \R \times S\M,
\end{equation*}
and the conditions
\begin{equation*}\label{condition alpha}
 \alpha(0,x,\theta)=0 \mbox{ and } \alpha(\pm T,x,\theta)=0,\;\; \forall(x,\theta)\in  S\M.
 \end{equation*}
\end{lemma}
\begin{proof}
 Let $(t,x,\theta)\in \R \times S\M$. We have
 \begin{align*}
\D  \alpha(t,x,\theta)&=\D \big( \phi(\gamma_{x,\theta}(-t),\dot{\gamma}_
{x,\theta}(-t))\big)\cr
&=\quad \p_s \big[\phi (\gamma_{\gamma_{x,\theta}(s),\dot{\gamma}_
{x,\theta}(s)}(-t),\dot{\gamma}_
{\gamma_{x,\theta}(s),\dot{\gamma}_
{x,\theta}(s)}(-t))\big]_{|s=0}.
\end{align*} Since  $\gamma_{\gamma_{x,\theta}(s),\dot{\gamma}_
{x,\theta}(s)}(-t)=\gamma_{x,\theta}(s-t)$ and $\dot{\gamma}_
{\gamma_{x,\theta}(s),\dot{\gamma}_
{x,\theta}(s)}(-t)=\dot{\gamma}_
{x,\theta}(s-t))$, we obtain
 \begin{align*}\label{D alpha}
\D \alpha(t,x,\theta)&=\quad \p_s \big[\phi(\gamma_{x,\theta}(s-t),\dot{\gamma}_
{x,\theta}(s-t))\big]_{|s=0}\cr
&=\quad \p_s \big[\alpha(t-s,x,\theta)\big]_{|s=0}
\cr
&=\quad -\p_t \alpha(t,x,\theta).
\end{align*}
With (\ref{phi}) and (\ref{condition T}), we get well
$\alpha(0,x,\theta)=\phi(x,\theta)=0$ and  $\alpha(\pm T,x,\theta)=\phi(\gamma_{x,\theta}(\pm T)$, $\dot{\gamma}_
{x,\theta}(\pm T))=0$, for all $(x,\theta) \in S\M$.
\end{proof}
\begin{lemma}\label{lemma beta}
We define
\begin{equation}\label{beta}
\beta_a(t,x,\theta)=\exp\bigg(-\displaystyle \int_{0}^{t}a\big(\gamma_{x,\theta}(-s),\dot{\gamma}_{x,\theta}(-s)\big) \;ds \bigg),\;\;(t,x,\theta)\in\R \times S\M.
\end{equation}
Then \begin{equation}\label{eq beta}
\p _t\beta_a(t,x,\theta)+\D \beta_a(t,x,\theta) +a(x, \theta) \beta_a (t,x,\theta)=0, \;\;(t,x,\theta)\in\R \times S\M.
\end{equation}
\end{lemma}
\begin{proof}
We have $\p_t\beta_a(t,x,\theta)=-a(\gamma_{x,\theta}(-t),\dot{\gamma}_{x,\theta}
(-t))\beta_a(t,x,\theta)$  and
 \begin{align*}\label{D beta}
\D \beta_a(t,x,\theta)&=\p_r \bigg[\exp\bigg(-\displaystyle \int_{0}^{t}a\big(\gamma_{\gamma_{x,\theta}(r),\dot{\gamma}_
{x,\theta}(r)}(-s),\dot{\gamma}_{\gamma_{x,\theta}(r),\dot{\gamma}_
{x,\theta}(r)}(-s)\big) ds\bigg)\bigg]_{|r=0} \cr
&=\quad \p_r \bigg[\exp\bigg(-\displaystyle \int_{0}^{t}a\big(\gamma_{x,\theta}(r-s),\dot{\gamma}_{x,\theta}
(r-s)\big) \;ds \bigg)\bigg]_{|r=0}\cr
&=\quad \p_r \bigg[ \exp\bigg(\displaystyle \int_{r}^{r-t}a\big(\gamma_{x,\theta}(s),\dot{\gamma}_{x,\theta}(s)\big) \;ds \bigg)\bigg]_{|r=0}\cr
&=\quad  \big[a\big(\gamma_{x,\theta}(-t),\dot{\gamma}_{x,\theta}(-t)\big)-
a(x,\theta) \big] \exp\bigg(\displaystyle \int_{0}^{-t}a\big(\gamma_{x,\theta}(s),\dot{\gamma}_{x,\theta}(s) \big) \;ds \bigg)\cr
&=\quad  \big[a\big(\gamma_{x,\theta}(-t),\dot{\gamma}_{x,\theta}(-t)\big)-
a(x,\theta) \big]\beta_a(t,x,\theta).
\end{align*}
Thus, the transport equation (\ref{eq beta}) is well proved.
\end{proof}
Now, for $\lambda > 0$, consider the function
\begin{equation}\label{psi lambda}
 \psi_{ \lambda}(t,x,\theta)=e^{ i \lambda(t+\tau_-(x,\theta))},\;\;(t,x,\theta)\in \R \times S\M.
 \end{equation}
 We compute seemingly and we use the property $\tau_-(\gamma_{x,\theta}(s),\dot{\gamma}_
{x,\theta}(s))=\tau_-(x,\theta)-s$, then we get
\begin{equation*}\label{ D psi lambda}
\D  \psi_{ \lambda}(x,\theta) =- i \lambda  \psi_{ \lambda}(x,\theta).
\end{equation*} Combining this with the two previous lemmas, we infer that \begin{equation*}\label{alpha beta psi}
    \big( \p_t (\alpha \beta_a \psi_\lambda)+\D (\alpha \beta_a \psi_\lambda)+a(x,\theta) (\alpha \beta_a \psi_\lambda)\big)(t,x,\theta)=0,\;\; (t,x,\theta)\in \R \times S\M.
    \end{equation*}
    This leads to define the remainders $r_{j,\lambda},\;j=1,2$, in the expressions of $u_1$ and $u_2$, as the respective unique solutions of the initial and final boundary value problems  defined as follows:
     \begin{equation}\label{r1 lamda}
\left\{
\begin{array}{lll}
\bigg(\partial_t r_{1,\lambda}+\D r_{1,\lambda}+a(x,\theta)r_{1,\lambda}\bigg)(t,x,\theta)=\bigg(
\I_k[r_{1,\lambda}]+\I_k[\alpha_1 \beta_a \psi_{\lambda}]\bigg)(t,x,\theta)  & \textrm{in }\; Q_T,\cr
r_{1,\lambda}(0, \cdot,\cdot )=0 & \textrm{in }\; S\M ,\cr
r_{1,\lambda}=0 & \textrm{on } \; \Sigma_-,
\end{array}
\right.
\end{equation}
and \begin{equation}\label{r2 lamda}
\left\{
\begin{array}{lll}
\bigg(\partial_t r_{2,\lambda}+\D r_{2,\lambda}-a(x,\theta)r_{2,\lambda}\bigg)(t,x,\theta)=\bigg(
\I^*_k[r_{2,\lambda}]+\I^*_k[\alpha_2 \beta_{-a} \psi_{-\lambda}]\bigg)(t,x,\theta)  & \textrm{in }\; Q_T,\cr
r_{2,\lambda}(T, \cdot,\cdot )=0 & \textrm{in }\; S\M ,\cr
r_{2,\lambda}=0 & \textrm{on } \; \Sigma_+.
\end{array}
\right.
\end{equation}
For $j=1,2$, we let $\phi_j\in C_0^{\infty}(\widetilde{\M}\setminus \M,C(\mathbb{S}^{n-1}))$ and we define
\begin{equation}\label{phi1 phi2}
 \alpha_j(t,x,\theta)=\phi_j(\gamma_{x,\theta}(-t),\dot{\gamma}_
{x,\theta}(-t)). \end{equation}
From Lemma \ref{lemma2}, there exist a constant $C$ such that \begin{equation}\label{estimate r1 Ik}
  \norm{r_{1,\lambda}(t,\cdot,\cdot)}_{L^p(S\M)}\leq  C \norm{\I_k[\alpha_1 \beta_a \psi_{\lambda}]}_{L^p(Q_T)},\;\forall t\in [0,T].
\end{equation}
and \begin{equation}\label{estimate r2 Ik*}
  \norm{r_{2,\lambda}(t,\cdot,\cdot)}_{L^p(S\M)}\leq  C \norm{\I^*_k[\alpha_2 \beta_{-a}\psi_{-\lambda}]}_{L^p(Q_T)},\;\forall t\in [0,T].
\end{equation}
We need the following result.
\begin{lemma}\label{lemma Ik Ik*}
   Let $a\in  \mathscr{A}_0(C_0)$ and $k\in  \mathscr{K}(C_1,C_2)$. There exist a constant $C>0$ depending only on $T,\M,C_0,C_1$ and $C_2$, such that for any  $p \geq 1$, and any $t\in [0,T]$, we have
   \begin{equation*}\label{estimate Ik phi}
   \norm{\I_k[\alpha_1 \beta_a\psi_{\lambda}](t,\cdot,\cdot)}_{L^p(S\M)}\leq C \norm{\phi_1}_{L^p(S\widetilde{\M})},\;  \norm{\I^*_k[\alpha_2 \beta_{-a} \psi_{-\lambda}](t,\cdot,\cdot)}_{L^p(S\M)}\leq C \norm{\phi_2}_{L^p(S\widetilde{\M})}.
   \end{equation*}
   Moreover, we have  \begin{equation}\label{lim Ik Ik*}
  \displaystyle \lim_{\lambda \rightarrow \infty} \norm{\I_k[\alpha_1 \beta_a\psi_{\lambda}]}_{L^2(Q_T)}= \displaystyle \lim_{\lambda \rightarrow \infty}  \norm{\I^*_k[\alpha_2 \beta_{-a}\psi_{-\lambda}]}_{L^2(Q_T)}=0.
   \end{equation}
\end{lemma}
\begin{proof}
Since $a\in  \mathscr{A}_0(C_0)$, we have $$
\abs{\I_k[\alpha_1 \beta_a\psi_{\lambda}](t,\cdot,\cdot)}\leq e^{C_0T} \int_{S_x\M}\abs{k(x,\theta,\theta')\alpha_1(t,x,\theta')} \dd\omega_x( \theta').$$ Let $p'$ be the integer satisfying $\frac{1}{p}+\frac{1}{p'}=1$. Applying the H\"older's inequality, we get
 \begin{align*}
\abs{\I_k[\alpha_1 \beta_a\psi_{\lambda}](t,\cdot,\cdot)}&\leq e^{C_0T} \big( \int_{S_x\M}\abs{k(x,\theta,\theta')}\dd\omega_x( \theta')\big)^{\frac{1}{p'}} \big( \int_{S_x\M}\abs{k(x,\theta,\theta')  \alpha_1^p(t,x,\theta')}  \dd\omega_x( \theta')\big)^{\frac{1}{p}}\cr
&\leq \quad  e^{C_0T}C_1^{\frac{1}{p'}} \big( \int_{S_x\M}\abs{k(x,\theta,\theta')}\abs{ \alpha_1^p (t,x,\theta')}\dd\omega_x( \theta')\big)^{\frac{1}{p}}.
\end{align*}
So $\norm{\I_k[\alpha_1\beta_a\psi_{\lambda}](t,\cdot,\cdot)}_{L^p(S\M)}
\leq e^{C_0T}C_1^{\frac{1}{p'}}C_2^{\frac{1}{p}}\norm{\alpha_1(t,\cdot,\cdot)
}_{L^p(S\M)}.$ Since $\alpha_1$ is defined by (\ref{phi1 phi2}), we obtain  \\ $\norm{\I_k[\alpha_1 \beta_a\psi_{\lambda}](t,\cdot,\cdot)}_{L^p(S\M)}\leq C \norm{\phi_1}_{L^p(S\widetilde{\M})}.$
On the same way, we prove the analogous estimate for $\I_k^*$.\\
To prove (\ref{lim Ik Ik*}), we put $\tilde{k}_{t,x}(\theta,\theta')=k(x,\theta,\theta') (\alpha_1 \beta_a)(t,x,\theta')$, then
\begin{align*}e^{-i\lambda t}\I_k[\alpha_1 \beta_a\psi_{\lambda}](t,x,\theta)=&\int_{S_x\M}k(x,\theta,\theta') (\alpha_1 \beta_a)(t,x,\theta') e^{i\lambda \tau_-(x,\theta')}\dd\omega_x( \theta')\cr =&
\int_{S_x\M}\tilde{k}_{t,x}(\theta,\theta') e^{i\lambda \tau_-(x,\theta')}\dd\omega_x( \theta').
\end{align*}
Since $k(x,\cdot,\cdot) \in L^2(S_x\M,S_x\M)$, $\int_{S_x\M}k(x,\theta,\theta')\dd\omega_x( \theta')\leq C_1$ and $\int_{S_x\M}k(x,\theta,\theta')\dd\omega_x( \theta)\leq C_2$  then for any $x\in \M$, the integral operator defined by
\begin{equation*}
 \begin{array}{lll}
K_x:  & L^2(S_x\M)& \longrightarrow L^2(S_x\M),\cr
 & f_x & \longmapsto K_x(f_x)(\theta)= \int_{S_x\M}\tilde{k}_{t,x}(\theta,\theta')f_x(\theta')
\dd\omega_x( \theta')
\end{array}
 \end{equation*} is compact in $L^2(S_x\M)$. This implies that the operator
 \begin{equation*}
\begin{array}{lll}
K:  & L^2(S\M)& \longrightarrow L^2(S\M),\cr
& f& \longmapsto K(f)(x,\theta)= K_x(f_x)(\theta)= \int_{S_x\M}\tilde{k}_{t,x}(\theta,\theta')f(x,\theta')\dd\omega_x( \theta')
\end{array}
\end{equation*}  is compact in $L^2(S\M)$. So the fact that $(x,\theta) \mapsto e^{i\lambda \tau_-(x,\theta)}$ converges weakly towards zero in the Hilbert space $L^2(S\M)$ ( see Lemma  \ref{lemA2}) leads to
$$\norm{e^{-i\lambda t}\I_k[\alpha_1 \beta_a\psi_{\lambda}](t,\cdot,\cdot)}_{L^2(S\M)} \longrightarrow 0,\; \mbox{ as } \lambda \longrightarrow \infty.$$ Hence $\norm{e^{-i\lambda t}\I_k[\alpha_1 \beta_a\psi_{\lambda}](t,\cdot,\cdot)}_{L^2(S\M)} $ is bounded independently of $\lambda$ and in light of the dominated convergence Theorem, we get $$\norm{\I_k[\alpha_1 \beta_a\psi_{\lambda}]}_{L^2(Q_T)}=\displaystyle \lim_{\lambda \rightarrow \infty}\norm{e^{-i\lambda t}\I_k[\alpha_1 \beta_a\psi_{\lambda}]}_{L^2(Q_T)}=0.$$
Seemingly, we prove that $\displaystyle \lim_{\lambda \rightarrow \infty}  \norm{\I^*_k[\alpha_2 \beta_{-a}\psi_{-\lambda}]}_{L^2(Q_T)}=0.$
\end{proof}
Combining Lemma \ref{lemma Ik Ik*} with (\ref{estimate r1 Ik}) and (\ref{estimate r2 Ik*}), we end up with the following result.
\begin{lemma}\label{uj estimate rj phij}
  Let $p \geq 1,\,a\in  \mathscr{A}_0(C_0),\,k\in  \mathscr{K}(C_1,C_2)$. For any $\lambda > 0$ and any $\phi_j\in C_0^{\infty}(\widetilde{\M}\setminus \M,C(\mathbb{S}^{n-1})),\; j=1,2$, each of the systems
    \begin{equation}\label{cauchy u1}
\left\{
\begin{array}{lll}
\partial_t u_1(t,x,\theta)+\D u_1(t,x,\theta)+a(x,\theta)u_1(t,x,\theta)=
\I_k[u_1](t,x,\theta)  & \textrm{in }\; Q_T,\cr
u_1(0, \cdot,\cdot )=0 & \textrm{in }\; S\M,
\end{array}
\right.
\end{equation}
 \begin{equation}\label{cauchy u2}
\left\{
\begin{array}{lll}
\partial_t u_2(t,x,\theta)+\D u_2(t,x,\theta)-a(x,\theta)u_2(t,x,\theta)=
\I^*_k[u_2](t,x,\theta)  & \textrm{in }\; Q_T,\cr
u_2(T, \cdot,\cdot )=0 & \textrm{in }\; S\M,
\end{array}
\right.
\end{equation}
has a unique solution in $C^1([0,T];L^p(S\M))\cap  C^0([0,T];\mathcal{W}^p)$, of the form $$ u_1(t,x,\theta)=(\alpha_1 \beta_a)(t,x,\theta) e^{i\lambda(t+\tau_-(x,\theta))}+r_{1,\lambda} (t,x,\theta)
  $$ and  $$u_2(t,x,\theta)=(\alpha_2 \beta_{-a})(t,x,\theta) e^{-i\lambda(t+\tau_-(x,\theta))}+r_{2,\lambda} (t,x,\theta).$$
 where \begin{equation}\label{rj tend vers 0} \displaystyle \lim_{\lambda \rightarrow \infty}\norm{r_{j,\lambda}}_{L^2(Q_T)}=0,
  \end{equation}
  and there exist a constant $C>0$ depending on $\M,T,C_0,C_1,C_2$ such that for $j=1,2$, we have
 \begin{equation}\label{estimate rj phij}
 \norm{r_{j,\lambda}}_{C^0([0,T];L^p(S\M))}\leq C \norm{\phi_j}_{L^p(S\widetilde{\M} )}. \end{equation}
\end{lemma}
\section{Recovery of the absorption coefficient}
\setcounter{equation}{0}
We will use the geometric optics solutions constructed in the previous section to prove that we can determine the absorption coefficient $a(x,\theta)$, up to a gauge transformation, from the albedo operator $\mathcal{A}_{a,k}$. When $a$ depends only on the spatial coordinate $x$, we will show that the albedo operator determines stably the geodesic ray transform of $a$. Then we use the invertibility of this transform, for simple manifolds, to conclude with a stability estimate reconstructing the absorption coefficient $a$ from  $\mathcal{A}_{a,k}$.
\subsection{Preliminaries}
\subsubsection*{Horizontal and vertical covariant derivatives}
  Let $(x^1,\dots,x^n)$ be a local coordinates system in $\M$. Then  any vector field $X \in \mathcal{C}^{\infty}(\M)$ can be uniquely represented as $X=X^i \frac{\p}{\p x^{i}}.$ We abbreviate on the following way: $X=(X^i).$
The terms $X^i$ are called the coordinates of the field $X$ in the given coordinates system.
Let us define the covariant differentiation of smooth  functions and vector fields on $\M$. Denote by  $\Gamma_{kq}^p$  the Christoffel symbols. For a smooth function $v$ on $\M$, we define the field   $\nabla v$ by $$\nabla v=\nabla_kv  dx^{k},$$ where
\begin{equation*}
\nabla_kv=\frac{\p}{\p x^k}v.
\end{equation*}
 For a vector field  $X=(X^i)$ on $\M$, we define the field   $\nabla X$ by $$\nabla X=\nabla_kX^i \frac{\p}{\p x^{i}}\otimes dx^{k},$$ where
\begin{equation*}
\nabla_kX^i=\frac{\p}{\p x^k}X^i+\Gamma^{i}_{kp}X^p.
\end{equation*}
Next, we extend this covariant differentiation for smooth functions and  vector fields on $\M$ to functions and vector  fields defined  on $T\M$. This extension  gives rise to two covectors fields (one-forms).\\
%%%%%%%%%%%%%%%%%%%%%%%%%%%%%%%%%%%%%%%%%%%%%%%%%%%%%%%%%%%ù
For $u\in\mathcal{C}^\infty(T\M)$,  we define $\overset{\hh}{\nabla}u$ and $\overset{\vv}{\nabla}u$ the horizontal and vertical covariant derivatives of $u$   by
\begin{equation}\label{nabla h u}
\overset{\hh}{\nabla} u=(\overset{\hh}{\nabla}_k u)dx^k,\quad \overset{\hh}{\nabla}_k u=\frac{\p u}{\p x^k}-\Gamma_{kq}^p\theta^q\frac{\p u}{\p\theta^p},
\end{equation}
and
\begin{equation*}\label{nabla v u}
\overset{\vv}{\nabla}u=(\overset{\vv}{\nabla}_ku)dx^k,\quad \overset{\vv}{\nabla}_ku=\frac{\p u}{\p\theta^k}.
\end{equation*}
Notice that in local coordinates system, (\ref{D}) reads \begin{equation}\label{Du nabla u}
            \D u(x,\theta)=\theta^i\overset{\hh}{\nabla}_i u.
          \end{equation}
%%%%%%%%%%%%%%%%%%%%%%%%%%%%%%%%%%%%%%%%%%%
 For $X$ a smooth vector field on $T\M$, we define $\overset{\hh}{\nabla}X$  and $\overset{\vv}{\nabla}X$ by
\begin{equation*}\label{horizontal}
\overset{\hh}{\nabla}X =\overset{\hh}{\nabla}_k X^i\frac{\p}{\p \theta^{i}}\otimes  dx^k,\;\overset{\hh}{\nabla}_kX^{i} =\frac{\p}{\p x^k}X^i-\Gamma^{p}_{kq}\theta^q \frac{\p}{\p \theta^{p}}X^{i}+\Gamma^{l}_{kp}X^p,
\end{equation*}
and
\begin{equation*}\label{vertical}
\overset{\vv}{\nabla}X =\overset{\vv}{\nabla}_k X^i \frac{\p}{\p \theta^{i}}\otimes dx^k,\;\overset{\vv}{\nabla}_k X^{i}=\frac{\p }{\p \theta^k} X^i.
\end{equation*}
The differential operators $ \overset{\hh}{\nabla}$ and $\overset{\vv}{\nabla}$ are respectively called the \textit{horizontal} and \textit{vertical} covariant derivatives. For more details, we refer to \cite{[Sh1]} Theorem 3.2.1 p 85. We have the following properties ( \cite{[Sh1]}, p 95):
\begin{equation*}\label{nabla k theta i}
\overset{\vv}{\nabla}_k \overset{\hh}{\nabla}_l=\overset{\hh}{\nabla}_l\overset{\vv}
{\nabla}_k,\quad \overset{\hh}{\nabla}_k\theta^i=0 \mbox{ and } \overset{\vv}{\nabla}_k\theta^i=\delta_k^i.
\end{equation*}
Now, we define the horizontal divergence $\overset{\hh}{\dive}$ and the vertical divergence $\overset{\vv}{\dive}$ of a vector field $X$ on $T\M$ by
\begin{equation}\label{def div}
 \overset{\hh}{\dive}(X)= \overset{\hh}{\nabla}_k X^k,\quad \overset{\vv}{\dive}(X)= \overset{\vv}{\nabla}_k X^k.
\end{equation}
In particular, we have  \begin{equation}\label{div theta}
  \overset{\hh}{\dive}(\theta)= \overset{\hh}{\nabla}_k \theta^k=0.
\end{equation}
%For a  vector field $X$, we have the following divergence formula ( cite{[Sh]}, p 101). \\% For $k+n-1\neq 0$ we have
%the \textit{Gauss-Ostrogradskii formula} of the vertical divergence
%\begin{equation}\label{3.27}
%\int_{S\M} \overset{\vv}{\dive}(X)\, \dvv=(k+n-1) \int_{S\M} \langle X,\theta\rangle \, \dvv,\quad X\in \mathcal{C}^\infty(T\M).
%\end{equation}
We have the following \textit{Gauss-Ostrogradskii formula} for  the horizontal divergence (\cite{[Sh1]},(3.6.31)):
\begin{equation}\label{divergence formula horizontal}
\int_{S\M} \overset{\hh}{\dive}(X)\, \dvv= \int_{\p S\M} \langle X,\nu\rangle \, \dss, \quad X\in \mathcal{C}^\infty(T\M).
\end{equation}
%%%%%%%%%%%%%%%%%%%%%%%%%%%%%%%%%%%%%%%%%%%%%%%%%%%%%%%%%%%%%%
 \subsection{Proof of Theorem \ref{th0}}
In fact, we will show the estimate (\ref{ineq -T T}) for $(x,\theta)\in S \widetilde{\M}$. Recall that we have extended $\M$ to a simple compact Riemannian manifold $ \widetilde{\M}$:
\begin{equation*}
  \widetilde{\M}\Supset \M \mbox{ such that } \forall x \in \widetilde{\M}\setminus \M,\;\;dist_{\g}(x,\M) < \epsilon.
  \end{equation*}
Take $k_1,k_2\in\mathscr{K}(C_1,C_2)$ and
$a_1,a_2\in\mathscr{A}_0(C_0)$ satisfying the condition (\ref{cond th0}), that is \begin{equation*}
              \p^\alpha_{x}a_1(x,\theta)= \p^\alpha_{x}a_2(x,\theta), \;x\in \p S\M, \alpha\in \N^n,\abs{\alpha} \leq 1.
 \end{equation*}

We extend $a_{1},a_{2}$ in  $C^{1}(S \widetilde{\M})$ in such a way that $a_1(x,\theta)=a_2(x,\theta),\;\forall x\in \widetilde{\M }\setminus \M,\; \forall \theta \in \mathbb{S}^{n-1}$. As mentioned earlier, the extension of $a_{1},a_{2}$ to  $S\widetilde{\M}$ will also be denoted  by $a_{1},a_{2}$. For all $x\in \widetilde{\M}$,  $(\theta,\theta')\in S_x\widetilde{\M} \times S_x\widetilde{\M}$, we set $$a(x,\theta)=(a_1-a_2)(x,\theta),\quad k(x,\theta,\theta')=(k_1-k_2)(x,\theta,\theta').$$
 For $j=1,2$, we consider a function  $\phi_j$ satisfying (\ref{phi}) and we Let $\alpha_j$ be defined by (\ref{phi1 phi2}) and $\beta_{a_j}$ like in  (\ref{beta})  by taking $a=a_j$.
We denote
\begin{equation*}\label{beta a1-a2}
\beta_{a}(t,x,\theta)=\beta_{a_1}(t,x,\theta)\beta_{-a_2}
(t,x,\theta).
\end{equation*}
Under these hypothesis, we have the following preliminary result.
\begin{lemma}\label{lemma3.1}
 We have \begin{equation}\label{estm lem3.1}
 \bigg| \int_0^T\!\!\!\int_{S\M}  a(x,\theta) (\alpha_1\alpha_2\beta_a)(t,x,\theta) \, \dvv \,dt \bigg| \leq \norm{\mathcal{A}_{a_2,k_2}-\mathcal{A}_{a_1,k_1}} _1 \norm{\phi_{1}}_{L^1(S\widetilde{\M})}\norm{\phi_2}_
{\mathcal{L}^{\infty}(S\widetilde{\M})}.\end{equation}
\end{lemma}
\begin{proof}
Take $\lambda >0$,  Lemma \ref{uj estimate rj phij} guarantees
the existence of a geometric optics solution $u_1 \in C^1([0,T];L^p(S\M))\cap  C^0([0,T];\mathcal{W}^p)$ on the form
\begin{equation*}
u_1(x,t)=(\alpha_1\beta_{a_1})(x,t)e^{i\lambda (t+\tau_-(x,\theta))}+r_{1,\lambda}(t,x,\theta),
\end{equation*}
to the following Boltzmann equation with initial data:
\begin{equation}\label{u1 u(0)}
\para{\p_t+\D +a_1(x,\theta)}u_1(t,x,\theta)=\I_{k_1}[u_1](t,x,\theta)\quad \textrm{in}\,Q_T, \quad u_1(0,\cdot,\cdot)=0 \quad \textrm{in}\,\, S\M,
\end{equation}
where $r_{1,\lambda}$ satisfies (\ref{r1 lamda}) and
\begin{gather}\label{norm r1 L2}
\norm{r_{1,\lambda}}_{C^0([0,T];L^p(S\M))}\leq C \norm{\phi_1}_{L^p(S\widetilde{\M} )},
\\ \nonumber
\displaystyle \lim_{\lambda \rightarrow \infty}\norm{r_{1,\lambda}}_{L^2(Q_T)}=0.
\end{gather}
Let us denote by $f_\lambda$ the function
$$
f_\lambda(t,x,\theta)=(\alpha_1\beta_{a_1})(x,t)e^{i\lambda (t+\tau_-(x,\theta))},\quad  (t,x,\theta)\in\Sigma_-,
$$
and consider $v$ the solution of the following non-homogenous boundary value problem
\begin{equation*}
\left\{\begin{array}{lll}
\para{\p_t+\D +a_2(x,\theta)}v(t,x,\theta)=\I_{k_2}[v], & \textrm{in}\,\, Q_T,\cr
v(0,x,\theta)=0, & \textrm{in}\,\, S\M,\cr
v(t,x,\theta)=u_1(t,x,\theta):=f_{\lambda}(t,x,\theta), & \textrm{on}\,\, \Sigma_-.
\end{array}
\right.
\end{equation*}
We let $ w=v-u_1$. Then, $w$ solves the following homogenous boundary value problem
\begin{equation}\label{syst w}
\left\{\begin{array}{lll}
\para{\p_t+\D +a_2(x,\theta)}w(t,x,\theta)=\I_{k_2}[w]- \I_{k}[u_1]+a u_1& \textrm{in}\,\,  Q_T,\cr
w(0,x,\theta)=0, & \textrm{in}\,\, S\M,\cr
w(t,x,\theta)=0, & \textrm{on}\,\,  \Sigma_-.
\end{array}
\right.
\end{equation}
Let $p'$ satisfying $\frac{1}{p}+\frac{1}{p'}=1$ then applying Lemma \ref{uj estimate rj phij}, we have a special solution
  $$ u_2 \in C^1([0,T];L^{p'}(S\M))\cap C^0([0,T];\mathcal{W}^{p'}),$$
to the backward Boltzmann equation
\begin{align}\label{syst u2}
\para{\p_t+\D -a_2(x,\theta)}u_2(t,x,\theta)=-\I^*_{k_2}[u_2](t,x,\theta)\quad \textrm{in}\,Q_T, \quad u_2(T,\cdot,\cdot)=0 \quad \textrm{in}\,\, S\M,
\end{align}
having the special form
\begin{equation*}
u_2(t,x,\theta)=(\alpha_2\beta_{-a_2})(t,x,\theta) e^{-i\lambda (t+\tau_-(x,\theta))}+r_{2,\lambda}(t,x,\theta),
\end{equation*}
where
\begin{gather}\label{syst r2}
\norm{r_{2,\lambda}}_{C^0([0,T];L^{p'}(S\M))}\leq C \norm{\phi_2}_{L^{p'}(S\widetilde{\M} )},
\\ \nonumber
\displaystyle \lim_{\lambda \rightarrow \infty}\norm{r_{2,\lambda}}_{L^2(Q_T)}=0.
\end{gather}
We first multiply the transport equation  in (\ref{syst w}) by $u_2$, then we integrate by parts over $Q_T$ by taking into account the conditions satisfied by $w$ and $u_2$ at the moments $0$ and $T$. Using the fact that we have
$$\int_0^T\!\!\!\int_{S\M} u_2  \I_{k_2}[w] \, \dvv dt \,=\int_0^T\!\!\!\int_{S\M} w\I_{k_2}^*[u_2] \, \dvv \,dt,
 $$ the fact that $u_2$ solves  (\ref{syst u2}) and the property $u_2 \D w=\D (wu_2)-w \D u_2$, we get
$$\int_0^T\!\!\!\int_{S\M} \big( a(x,\theta)u_2 u_1-u_2 \I_k[u_1] \big) \, \dvv \,dt=\int_0^T\!\!\!\int_{S\M} \D(wu_2)  \, \dvv \,dt.$$
We have $ \D (wu_2)(x,\theta)=\theta^i\overset{\hh}{\nabla}_i(wu_2)(x,\theta)$ (see (\ref{Du nabla u})). Using (\ref{def div}), we easily verify that $ \overset{\hh}{\dive}(f X)= (\overset{\hh}{\nabla}_i f) X^i+f \overset{\hh}{\dive}X$ where $f$ is a function and $X$ is a vector field on $S\M$. Then we use successively (\ref{div theta}), the Gauss-Ostrogradskii formula of the horizontal divergence (\ref{divergence formula horizontal}) and another time the first equality of  (\ref{syst u2}), to obtain $$ \int_0^T\!\!\!\int_{S\M} a(x,\theta)u_2 u_1-u_2 \I_k[u_1] \, \dvv (x,\theta) \,dt=\int_0^T\!\!\!\int_{\p S\M} \langle \theta,\nu(x) \rangle wu_2  \, \dss \,dt.$$
We have  $w=v-u_1=0$ on $\Sigma_-$ and $w_{|\Sigma_+}=v_{|\Sigma_+}-u_{1|\Sigma_+}=
\mathcal{A}_{a_2,k_2}[f_\lambda]-\mathcal{A}_{a_1,k_1}[f_\lambda]$. So, we deduce that
\begin{multline}\label{equation integr1}
\int_0^T\!\!\!\int_{S\M} \big(a(x,\theta)u_1u_2-u_2\I_k[u_1]\big) \, \dvv \,dt
=\int_0^T\!\!\!\int_{\p_+S\M} u_2 \para{\mathcal{A}_{a_2,k_2}-\mathcal{A}_{a_1,k_1}} (f_{\lambda})  \,\langle \theta,\nu \rangle  \dss\,dt .
\end{multline}
Now, we shall estimate each one of the terms in the left hand  side member of the last equality.
Replacing $u_1$ and $u_2$ by their expressions, it follows that
\begin{equation}\label{member1}
\int_0^T\!\!\!\int_{S\M} a(x,\theta) u_1u_2(t,x,\theta) \, \dvv \,dt=
\int_0^T\!\!\!\int_{S\M} a(x,\theta)\alpha_1\alpha_2\beta_a(t,x,\theta) \, \dvv\,dt+\mathscr{R}_{1,\lambda},
\end{equation}
where \begin{multline*}
\mathscr{R}_{1,\lambda}:=\int_0^T\!\!\!\int_{S\M} a(x,\theta) \alpha_1 \beta_{a_1}\psi_{\lambda}r_{2,\lambda}(t,x,\theta) \, \dvv \,dt+\cr \int_0^T\!\!\!\int_{S\M}\big( a(x,\theta) \alpha_2 \beta_{-a_2}\psi_{-\lambda}r_{1,\lambda}(t,x,\theta) +a(x,\theta) r_{1,\lambda}r_{2,\lambda}(t,x,\theta) \big) \, \dvv \,dt.
\end{multline*}
Here $\psi_{\lambda}$ is defined by (\ref{psi lambda}). So, there exist a constant $C$ depending on $C_0$ and $T$ that verifies $$\abs{\mathscr{R}_{1,\lambda}}\leq C\big(\norm{r_{1,\lambda}}_{L^2(Q_T)}+\norm{r_{2,\lambda}}_{L^2(Q_T)} \big).$$
Similarly, we have  \begin{equation}\label{member2}
\int_0^T\!\!\!\int_{S\M} u_2\I_k[u_1](t,x,\theta) \, \dvv \,dt=
\int_0^T\!\!\!\int_{S\M} \alpha_2\beta_{-a_2}\psi_{-\lambda} \I_k[ \alpha_1\beta_{a_1}\psi_{\lambda}](t,x,\theta) \, \dvv \,dt+\mathscr{R}_{2,\lambda},
\end{equation}
where \begin{multline*}
\mathscr{R}_{2,\lambda}:=\int_0^T\!\!\!\int_{S\M} \alpha_2 \beta_{-a_2}\psi_{-\lambda}\I_k[r_{1,\lambda}](t,x,\theta) \, \dvv \,dt \cr + \int_0^T\!\!\!\int_{S\M}\big( r_{2,\lambda} \I_k[ \alpha_1\beta_{a_1}\psi_{\lambda}](t,x,\theta) +r_{2,\lambda}\I_k[r_{1,\lambda}](t,x,\theta)\big) \, \dvv \,dt,
\end{multline*} and there exist a constant $C$ depending on $C_0$ and $T$ such that $$\abs{\mathscr{R}_{2,\lambda}}\leq C\big(\norm{r_{1,\lambda}}_{L^2(Q_T)}+\norm{r_{2,\lambda}}_{L^2(Q_T)} \big).$$
 In light of (\ref{norm r1 L2}) and (\ref{syst r2}), we have \begin{equation*}\label{lim R1 R2}  \displaystyle \lim_{\lambda \rightarrow \infty} \mathscr{R}_{1,\lambda}=\lim_{\lambda \rightarrow \infty} \mathscr{R}_{2,\lambda}=0.\end{equation*}
Applying the Cauchy-Schwartz inequality and (\ref{lim Ik Ik*}), we obtain that $$\displaystyle \lim_{\lambda \rightarrow \infty}\int_0^T\!\!\!\int_{S\M} \alpha_2\beta_{-a_2}\psi_{-\lambda} \I_k[ \alpha_1\beta_{a_1}\psi_{\lambda}](t,x,\theta) \, \dvv\,dt=0.$$
Introducing (\ref{member1}) and (\ref{member2}) in the equation (\ref{equation integr1}), we end up with
 \begin{multline}\label{equation integr2}
 \int_0^T\!\!\!\int_{S\M}  a(x,\theta) (\alpha_1\alpha_2\beta_a)(t,x,\theta) \, \dvv \,dt= \cr \int_0^T\!\!\!\int_{\p_+S\M} u_2 (t,x,\theta) \para{\mathcal{A}_{a_2,k_2}-\mathcal{A}_{a_1,k_1}} (f_{\lambda}) (t,x,\theta) \,\langle \theta,\nu(x) \rangle  \dss\,dt+ \mathscr{R}_{\lambda}.
\end{multline}
Here, $$\mathscr{R}_{\lambda}=-\mathscr{R}_{1,\lambda}+\mathscr{R}_{2,
\lambda}+\int_0^T\!\!\!\int_{S\M} \alpha_2\beta_{-a_2}\psi_{-\lambda} \I_k[ \alpha_1\beta_{a_1}\psi_{\lambda}](t,x,\theta) \, \dvv \,dt,$$ and we have  $$\displaystyle \lim_{\lambda \rightarrow \infty} \mathscr{R}_{\lambda}=0.$$
%Let us estimate the integral term in the right side of the equation (\ref{equation integr2}).
On the other hand and since the albedo operator is bounded, we have \begin{multline}\label{terme bord}\bigg| \int_0^T\!\!\!\int_{\p_+S\M} u_2 (t,x,\theta) \para{\mathcal{A}_{a_2,k_2}-\mathcal{A}_{a_1,k_1}} (f_{\lambda}) (t,x,\theta) \,\langle \theta,\nu(x) \rangle  \dss\,dt \bigg|\cr\leq \norm{\para{\mathcal{A}_{a_2,k_2}-\mathcal{A}_{a_1,k_1}} (f_{\lambda})}_{\mathcal{L}^p(\Sigma_+)}\norm{u_2}_{\mathcal{L}^{p'}
(\Sigma_+)} \cr \leq \norm{\mathcal{A}_{a_2,k_2}-\mathcal{A}_{a_1,k_1}} _p \norm{f_{\lambda}}_{\mathcal{L}^p(\Sigma_-)}\norm{u_2}_{\mathcal{L}
^{p'}(\Sigma_+)},\end{multline} where $\frac{1}{p}+\frac{1}{p'}=1$. By the boundary condition $r_2=0$ on $\Sigma_+$ in (\ref{r2 lamda}) and the definition of $\alpha_j,\;j=1,2$ given in (\ref{phi1 phi2}), we deduce that
\begin{align}\label{uj phij}
\norm{u_2}_{\mathcal{L}^{p'}(\Sigma_+)}\leq C \norm{\phi_2}_{L^{p'}(S\widetilde{\M})}\leq C \norm{\phi_2}_{L^{\infty}(S\widetilde{\M})},\quad \norm{f_{\lambda}}_{\mathcal{L}^p(\Sigma_-)}\leq  C \norm{\phi_1}_{L^{p}(S\widetilde{\M})}.
\end{align} Note that for $1<p<2$, we get by the Riesz-Thorin's Theorem $$\norm{\mathcal{A}_{a_2,k_2}-\mathcal{A}_{a_1,k_1}} _p \leq \norm{\mathcal{A}_{a_2,k_2}-\mathcal{A}_{a_1,k_1}} _1^{1-\kappa}\norm{\mathcal{A}_{a_2,k_2}-\mathcal{A}_{a_1,k_1}} _2^\kappa,$$ with $\kappa=2-\frac{2}{p}$. Hence $$\displaystyle \lim \sup_{p \rightarrow 1}\norm{\mathcal{A}_{a_2,k_2}-\mathcal{A}_{a_1,k_1}} _p \leq \norm{\mathcal{A}_{a_2,k_2}-\mathcal{A}_{a_1,k_1}} _1.$$
We also have $\displaystyle \lim_{p \rightarrow 1}\norm{f_{\lambda}}_{\mathcal{L}^p(\Sigma_-)}=\norm{f_{\lambda}}_
{\mathcal{L}^1(\Sigma_-)}$, $\displaystyle \lim_{p \rightarrow 1}\norm{\phi_1}_{L^{p}(S\widetilde{\M})}=\norm{\phi_1}_
{L^{1}(S\widetilde{\M})}$, then (\ref{equation integr2}) yields to the claimed estimate (\ref{estm lem3.1}).
This completes the proof of the Lemma.
\end{proof}

%\begin{lemma}\label{lem2}
   %Let  $a_1, a_2 \in \mathscr{A}(C_0,s)$ and $k_1, k_2 \in \mathscr{K}(C_{1},C_2)$, such that $\p_x^\alpha a_1=\p_x^\alpha a_2,\; x\in \p \M,\; \alpha \in \N^n,\; \abs{\alpha}\leq 1$. There exists a constant $C>0$ such that for any $\rho_j\in C_0^{\infty}(\tilde{\M}\setminus \M)$, $j\in\{1,2\}$ we have \begin{multline}\abs{ \int_{\tilde{\M}\setminus \M} \,(\rho_1 \rho_2)(x) \bigg[1-\exp  \big(-\displaystyle \int_{0}^{T}a(\gamma_{x,\xi}(s)) \;ds\big)  \bigg] \,\mu(x',\theta')\; dl\dss\cr\leq \norm{\mathcal{A}_{a_2,k_2}-\mathcal{A}_{a_1,k_1}} _1 \norm{\rho_{1}}_{L^1(\tilde{\M})}\norm{\rho_2}_
%{\mathcal{L}^{\infty}(\tilde{\M})}.\end{multline}
%\end{lemma}
%%%%%%%%%%%%%%%%%%%%%%%%%%%%
To prove the estimate (\ref{ineq -T T}) of Theorem \ref{th0}, we make a suitable choice of the functions $\alpha_1$ and $\alpha_2$ to introduce in (\ref{estm lem3.1}). We will use an approximation of the identity on the boundary sphere. More precisely, we will consider the Poisson kernel of $B(0,1)\subset T_{x} \widetilde{\M}$. Any other approximation may agree.
\smallskip
We define the Poisson kernel in $B(0,1)\subset T_{x} \widetilde{\M}$ by setting
$$
P(\xi,\theta)=\frac{1-\abs{\xi}^2}{\omega_n\abs{\xi-\theta}^n},\quad \xi\in B(0,1);\,\, \theta\in S_{x} \widetilde{\M}.
$$ Here $\omega_n$ is the volume of the unit ball of $\R^n$.
For $0<h<1$, we define $\Psi_h: S_{x} \widetilde{\M}\times S_{x} \widetilde{\M} \to \R$ as
\begin{align}\label{psi h}
\Psi_h(\xi,\theta)=P(h \xi,\theta).
\end{align}

\begin{lemma}\label{properties psi h}
Let $x\in \widetilde{\M}$, we have the following properties:
\begin{equation}\label{prop1}
\displaystyle 0\leq\Psi_h(\xi,\theta)\leq \frac{2}{\omega_n(1-h)^{n-1}},\quad \forall\, h\in (0,1),\,\forall\,\xi,\theta \in S_{x} \widetilde{\M}.
\end{equation}
\begin{equation}\label{prop2}
 \displaystyle \int_{S_{x} \widetilde{\M}}\Psi_h(\xi,\theta) d\omega_{x}(\theta)=1,\quad \forall h\in (0,1),\,\forall\,\xi \in S_{x} \widetilde{\M}.
\end{equation}
%\begin{equation}\label{prop3}
%\int_{S_{x'} \tilde{\M}}\Psi_\kappa(\xi,\theta) \abs{\xi-\theta}d\omega_{x'}(\theta)\leq C(1-\kappa)^{1/2n}, \quad \forall\, \kappa\in (0,1),\,\forall\,\xi \in S_{x'} \tilde{\M}.
%\end{equation}
For any $\upsilon \in C(S_{x} \widetilde{\M})$, we have the following  Poisson formula:
\begin{equation}\label{prop4}
\displaystyle \lim_{h \rightarrow 1}\int_{S_{x} \widetilde{\M}}\Psi_h(\xi,\theta) \upsilon(\theta)d\omega_{x}(\theta)=\upsilon(\xi),\,\forall\,\xi \in S_{x} \widetilde{\M},
\end{equation}
here the limit is taken in the topology of $L^p(S_{x} \widetilde{\M})$, $p\in [1,+\infty[$ and uniformly on $S_{x} \widetilde{\M}$.
\end{lemma}
We emphasize that when $x\in \widetilde{\M}\setminus \widetilde{\M}$, then $S_{x} \widetilde{\M}=\mathbb{S}^{n-1}$ and  $|\cdot |$ is the euclidian norm. In fact, in this section, we use this Lemma only for $x \in \widetilde{\M}\setminus \widetilde{\M}$. It is in the following section that we will use it on all $\widetilde{\M}$. We find the proof of this result in the Appendix \ref{proof psi}.\\
Now we are ready to prove Theorem \ref{th0}. We will proceed in two steps.
We start by proving  (\ref{ineq -T T}) for $x\in \widetilde{\M}\setminus \M,\; \xi\in \mathbb{S}^{n-1}$  and after for $(x,\xi)\in S\M$. \\
 Recall that we have extended $a$ by $0$ outside $\M$.  %Replacing $\alpha_1$, $\alpha_2$ and $\beta_a$ by their expression in the estimate of Lemma\ref{lem1},
Then Lemma \ref{lemma0} applied to the manifold $\widetilde{\M}$ gives \begin{multline*}
 \int_0^T\!\!\!\int_{S\M}  a(x,\theta) (\alpha_1\alpha_2\beta_a)(t,x,\theta) \, \dvv \,dt\cr
  =\int_0^T\!\!\!\int_{S\widetilde{\M}}  a(x,\theta) (\phi_1\phi_2)(\gamma_{x,\theta}(-t),\dot{\gamma}_
{x,\theta}(-t))\exp\bigg[-\displaystyle \int_{0}^{t}a(\gamma_{x,\theta}(-s),\dot{\gamma}_{x,\theta}(-s)) \;ds \bigg] \, \dvv \,dt\cr= \int_0^T\!\!\!\int_{\p_{-}S\widetilde{\M}} \int_{\R} a\big(\gamma_{x',\theta'}(r),\dot{\gamma}_{x',\theta'}(r)\big) (\phi_1\phi_2)(\gamma_{x',\theta'}(r-t),\dot{\gamma}_{x',\theta'}(r-t)) \cr\exp\bigg[-\displaystyle \int_{0}^{t}a\big(\gamma_{x',\theta'}(r-s),\dot{\gamma}_{x',\theta'}
(r-s)\big) \;ds \bigg] \,\mu(x',\theta')\; dr\dss \;dt .\end{multline*} To compress the  notation, we write $$\phi_{t}(x,\theta)=(\gamma_{x,\theta}(t),
\dot{\gamma}_{x,\theta}(t)),$$ where $\phi_{t}$ is the geodesic flow defined in (\ref{flowphit}).
We perform successively two variables changes. We set $l=r-t$, then we find \begin{multline*}
\int_0^T\!\!\!\int_{S\M}  a(x,\theta) (\alpha_1\alpha_2\beta_a)(t,x,\theta) \, \dvv\,dt
 = \int_0^T\!\!\!\int_{\p_{-}S\widetilde{\M}} \int_{\R} a\big(\phi_{l+t}(x',\theta')\big) (\phi_1\phi_2)(\phi_{l}(x',\theta'))
 \cr\exp\bigg[-\displaystyle \int_{0}^{t}a\big(\phi_{l+t-s}(x',\theta')\big) \;ds \bigg] \,\mu(x',\theta')\; dl\dss \;dt .
   \end{multline*} Then, by the variable change   $s'=t-s$, we get
 \begin{multline*}
 \int_0^T\!\!\!\int_{S\M}  a(x,\theta) (\alpha_1\alpha_2\beta_a)(t,x,\theta) \, \dvv\,dt
 =\int_0^T\!\!\!\int_{\p_{-}S\widetilde{\M}} \int_{\R} a\big(\phi_{l+t}(x',\theta')\big) (\phi_1\phi_2)(\phi_{l}(x',\theta')) \cr\exp\bigg[-\displaystyle \int_{0}^{t}a\big(\phi_{l+s'}(x',\theta')\big) \;ds'\bigg]  \,\mu(x',\theta')\; dl\dss \;dt \cr
 =\int_{\p_{-}S\widetilde{\M}} \int_{\R}\!\!\!  \, (\phi_1\phi_2)(\phi_{l}(x',\theta')) \displaystyle \int_0^T \!\!\! \frac{d}{dt}\bigg[-\exp  \bigg(-\displaystyle \int_{0}^{t}a\big(\phi_{l+s}(x',\theta')\big) \;ds\bigg)  \bigg ]\;dt \,\mu(x',\theta') \; dl\dss  \\
= \int_{\p_{-}S\widetilde{\M}}\int_{\R} \, (\phi_1\phi_2)(\phi_{l}(x',\theta')) \bigg[1-\exp  \bigg(-\displaystyle \int_{0}^{T}a\big(\phi_{l+s}(x',\theta')\big) \;ds\bigg)  \bigg] \,\mu(x',\theta')\; dl\dss
%=\int_{S\widetilde{\M}}\,(\phi_1\phi_2)(x,\theta) \bigg[1-\exp  \bigg(-\displaystyle \int_{0}^{T}\!\!\!  \, a\big(\phi_{l}(x,\theta )\big) \;ds\bigg)  \bigg] \, \, \dvv.
       \end{multline*}
Upon using Lemma \ref{lemma0}, we obtain
 \begin{multline*}
 \int_0^T\!\!\!\int_{S\M}  a(x,\theta) (\alpha_1\alpha_2\beta_a)(t,x,\theta) \, \dvv\,dt
 = \cr \int_{S\widetilde{\M}}\,(\phi_1\phi_2)(x,\theta) \bigg[1-\exp  \bigg(-\displaystyle \int_{0}^{T}\!\!\!  \, a\big(\gamma_{x,\theta}(s),\dot{\gamma}_
{x,\theta}(s))\big) \;ds\bigg)  \bigg] \, \, \dvv.
       \end{multline*}
Recall that we have extended the metric $\g$ by the identity outside $\M$.
 Due to our assumption on the metric; that is $(\widetilde{\M},\g)$ is simple, the following parallel translation map is globally well-defined: given
%$(x,\theta)\in S\widetilde{\M}$ and
 $x, y\in \widetilde{\M}$, we define $$\mathcal{P}(\cdot;x,y):S_x\widetilde{\M} \longrightarrow S_y\widetilde{\M}$$ such that $\mathcal{P}(\theta;x,y)$ is the parallel translation of $\theta$
along the (unique) geodesic joining $x$ and $y$.
 We take $x_0 \in  \widetilde{\M} \setminus \M$ and $\xi \in S_{x_0}\widetilde{\M}= \mathbb{S}^{n-1}$. We let $\Psi_h$ the positive function given by (\ref{psi h}). We choose $\rho_1,\rho_2\in C_0^{\infty}(\R^n)$ with $ \mathrm{supp}(\rho_j)\subset (\widetilde \M \setminus \M)$ and we set
 \begin{equation}\label{phij rhoj}
   \phi_1(x,\theta)=\rho_1(x)\Psi_h(\mathcal{P}(\xi;x_0,x),\theta) \mbox{ and  } \phi_2(x,\theta)=\rho_2(x).\end{equation}
Note that if $x\in \R^n\setminus \M$ then $\phi_1(x,\theta)=\rho_1(x)\Psi_h(\xi,\theta)$, $\forall \theta \in  \s^{n-1}$. Hence, we have
 \begin{multline*}
 \int_0^T\!\!\!\int_{S\M}  a(x,\theta) (\alpha_1\alpha_2\beta_a)(t,x,\theta) \, \dvv\,dt
 =\cr \int_{ \widetilde{\M}\setminus \M} \rho_1(x)\rho_2(x)  \int_{ \mathbb{S}^{n-1}} \, \Psi_h(\xi,\theta) \bigg[1-\exp  \big(-\displaystyle \int_{0}^{T}\, a\big(\gamma_{x,\xi}(s),\dot{\gamma}_{x,\xi}(s)\big)\, ds \big)  \bigg] \, \dvv.
\end{multline*}
We use (\ref{prop2}) to estimate $ \norm{\phi_{1}}_{L^1(S\widetilde{\M})}$  and we apply Lemma \ref{lemma3.1}, then we obtain  \begin{multline*}
  \bigg| \int_{\widetilde{\M}\setminus \M} \,\rho_1(x)\rho_2(x) \int_{ \mathbb{S}^{n-1}} \, \Psi_h(\xi,\theta) \bigg[1-\exp  \bigg(- \displaystyle \int_{0}^{T} \,a\big(\gamma_{x,\theta}(s), \dot{ \gamma}_{x,\theta}(s)\big)\bigg) \, ds \bigg]  \, \, \dvv \bigg| \cr\leq \norm{\mathcal{A}_{a_2,k_2}-\mathcal{A}_{a_1,k_1}} _1 \norm{\rho_{1}}_{L^1(\widetilde{\M})}\norm{\rho_2}_
{L^{\infty}(\widetilde{\M})} .\end{multline*}
The property (\ref{prop1}) shows that $\Psi_h(\xi,\theta)$ is bounded independently of $x$. So taking the limit when $h \rightarrow 1$ by applying the dominated convergence Theorem then using the Poisson formula (\ref{prop4}), we find \begin{multline*}
 \bigg| \int_{ \widetilde{\M}\setminus \M} \rho_1(x)\rho_2(x) \bigg[1-\exp  \big(-\displaystyle \int_{0}^{T}\, a\big(\gamma_{x,\xi}(s),\dot{\gamma}_{x,\xi}(s)\big)\, ds \big)  \bigg] \, \dvv \bigg|\\ \leq  \norm{\mathcal{A}_{a_2,k_2}-\mathcal{A}_{a_1,k_1}} _1 \norm{\rho_{1}}_{L^1(\widetilde{\M})}\norm{\rho_2}_
{L^{\infty}(\widetilde{\M})}
.\end{multline*}
For any $\xi \in \s^{n-1}$, the left side term in the last inequality is a bilinear form in $(\rho_1,\rho_2)\in L^1(\widetilde{\M}) \times L^{\infty}(\widetilde{\M})$ with norm $$\bigg\| 1-\exp  \big(-\displaystyle \int_{0}^{T}\, a(\gamma_{x,\xi}(s),\dot{\gamma}_{x,\xi}(s))\, ds \big)   \bigg\|_
{L^{\infty}(\widetilde{\M}\setminus \M)}.$$ Thus, for any $\xi \in \s^{n-1}$, we have
\begin{equation*}
\bigg\| 1-\exp  \big(-\displaystyle \int_{0}^{T}\, a(\gamma_{x,\xi}(s),\dot{\gamma}_{x,\xi}(s))\, ds \big)   \bigg\|_
{L^{\infty}(\widetilde{\M}\setminus \M)} \leq \norm{\mathcal{A}_{a_2,k_2}-\mathcal{A}_{a_1,k_1}} _1.\end{equation*}
Using the fact that $|X|\leq e^{m}\,|e^{X}-1|$ for any $|X|\leq m$, we get \begin{equation} \label{ineq[0,T]}
\Big|\displaystyle\int_{0}^T\!\!\! \,\,\!\!\! a(\gamma_{x,\xi}(s),\dot{\gamma}_{x,\xi}(s)) \;ds\Big| \leq C  \norm{\mathcal{A}_{a_2,k_2}-\mathcal{A}_{a_1,k_1}} _1,\, \forall x \in \widetilde{\M}\setminus \M,\, \xi \in \mathbb{S}^{n-1},\end{equation} where we take $X=\displaystyle \int_{0}^{T}\, a(\gamma_{x,\xi}(s),\dot{\gamma}_{x,\xi}(s))\, ds$.\\
Now let us prove that  \begin{equation} \label{ineq[-T,0]} \Big|\displaystyle\int_{-T}^0\!\!\! \,\,\!\!\! a(\gamma_{x,\xi}(s),\dot{\gamma}_{x,\xi}(s)) \;ds\Big| \leq C  \norm{\mathcal{A}_{a_2,k_2}-\mathcal{A}_{a_1,k_1}} _1,\,\forall x \in \widetilde{\M}\setminus \M,\, \xi \in \mathbb{S}^{n-1}. \end{equation}

 Since the metric of $\widetilde{\M}$ is simple, we have two cases: either $(\gamma_{x,\xi}(-s))_{s\geq 0}$ intersects or does not intersect $\M$.
At first, we suppose that for each  $ s \geq 0$, we have $ \gamma_{x,\xi}(-s) \notin \M$. Since we have extended $a$ by $0$ outside $\M$, we have  $\displaystyle\int_{-T}^0\!\!\! \,\,\!\!\! a(\gamma_{x,\xi}(s),\dot{\gamma}_{x,\xi}(s)) \;ds=\displaystyle\int_{0}^T\!\!\! \,\,\!\!\! a(\gamma_{x,\xi}(-s),\dot{\gamma}_{x,\xi}(-s)) \;ds=0$ and therefore the claimed inequality (\ref{ineq -T T}) holds true in this case.\\
For the second case, we assume that for some $s \geq 0$, $\gamma_{x,\xi}(-s) \in \M$. We define $$s_0:=\sup\{s \geq 0,\, \gamma_{x,\xi}(-s)\in \M \}.$$ Since $T >diam_g(\M)+2\epsilon$, we have $\displaystyle\int_{0}^T\!\!\! \,\,\!\!\! a(\gamma_{x,\xi}(-s),\dot{\gamma}_{x,\xi}(-s)) ds=\displaystyle\int_{0}^{s_0+\frac{\epsilon}{2}}\!\!\! \,\,\!\!\! a(\gamma_{x,\xi}(-s),\dot{\gamma}_{x,\xi}(-s)) ds$.  We make the variable change $s'=s_0+\frac{\epsilon}{2}-s$ and we set $(y,\theta)=(\gamma_{x,\xi}(-s_0-\frac{\epsilon}{2}),
      \dot{\gamma}_{x,\xi}(-s_0-\frac{\epsilon}{2}))$. Then  $\gamma_{y,\theta}(s_0+\frac{\epsilon}{2})=x$ and since we have (\ref{M tilde}) then $y \in \widetilde{\M}\setminus \M$ and we get $$\displaystyle\int_{0}^{s_0+\frac{\epsilon}{2}}\!\!\! \,\,\!\!\! a(\gamma_{x,\xi}(-s),\dot{\gamma}_{x,\xi}(-s)) \;ds=\displaystyle \int_{0}^{T}\!\!\! \,\,\!\!\! a(\gamma_{y,\theta}(s),\dot{\gamma}_{y,\theta}(s)) \;ds.$$ Therefore, we have  $$\displaystyle\int_{-T}^0\!\!\! \,\,\!\!\! a(\gamma_{x,\xi}(s),\dot{\gamma}_{x,\xi}(s)) \;ds=\displaystyle\int_{0}^{T}\!\!\! \,\,\!\!\! a(\gamma_{y,\theta}(s),\dot{\gamma}_{y,\theta}(s)) \;ds.$$
    Applying (\ref{ineq[0,T]}) for $y \in \widetilde{\M}\setminus \M,\,\theta \in \mathbb{S}^{n-1}$, we get (\ref{ineq[-T,0]}) and therefore (\ref{ineq -T T}) is proven for every $x \in  \widetilde{\M}\setminus \M$ and $\xi\in \mathbb{S}^{n-1}$.
Now we will  prove that (\ref{ineq -T T}) is also valid for $(x,\xi)\in S\M$. \\
We take  $(x,\xi)\in S\M$ then there exist a real $s\in \R$ such that $\gamma_{x,\xi}(s)\in \widetilde{\M}\setminus \M$. Let $s_1$ be a such real. We make a variable change by taking $s'=s-s_1$ then we set $(x_1,\xi_1)=(\gamma_{x,\xi}(s_1), \dot{\gamma}_{x,\xi}(s_1))$, we obtain $$\displaystyle\int_{-T}^T\!\!\! \,\,\!\!\! a(\gamma_{x,\xi}(s),\dot{\gamma}_{x,\xi}(s)) \;ds=\displaystyle\int_{-T-s_1}^{T-s_1}\!\!\! \,\,\!\!\! a(\gamma_{x,\xi}(s_1+s),\dot{\gamma}_{x,\xi}(s_1+s)) \;ds=\displaystyle\int_{-T-s_1}^{T-s_1}\!\!\! \,\,\!\!\! a(\gamma_{x_1,\xi_1}(s),\dot{\gamma}_{x_1,\xi_1}(s)) \;ds.$$ Since we have  $T>Diam_{\g} \M+2\epsilon $, then $\displaystyle\int_{-T}^T\!\!\! \,\,\!\!\! a(\gamma_{x,\xi}(s),\dot{\gamma}_{x,\xi}(s)) \;ds=\displaystyle\int_{-T}^{T}\!\!\! \,\,\!\!\! a(\gamma_{x_1,\xi_1}(s),\dot{\gamma}_{x_1,\xi_1}(s)) \;ds$. We apply (\ref{ineq -T T}), which is established for points in $\widetilde{\M}\setminus \M$ and their tangent vectors, to $(x_1,\xi_1)$ where $x_1 \in \widetilde{\M}\setminus \M$ and $\xi_1 \in  \mathbb{S}^{n-1}$. We infer that
$$\bigg|\displaystyle\int_{-T}^T\!\!\! \,\,\!\!\! a(\gamma_{x,\xi}(s),\dot{\gamma}_{x,\xi}(s)) \;ds \bigg| \leq C  \norm{\mathcal{A}_{a_2,k_2}-\mathcal{A}_{a_1,k_1}} _1,\,\forall(x,\xi)\in S\M.$$ This completes the proof of Theorem \ref{th0}.
\subsection{Uniqueness of the absorption coefficient modulo the gauge transformation}
This subsection is devoted to the proof of Corollary \ref{corol0}. We assume that
$\mathcal{A}_{a_2,k_2}=\mathcal{A}_{a_1,k_1}$. The Theorem \ref{th0} implies that \begin{equation}\label{int -TT=0} \displaystyle\int_{-T}^T\!\!\! \,\,\!\!\! a(\gamma_{x,\xi}(s),\dot{\gamma}_{x,\xi}(s)) \;ds =0,\,\forall(x,\xi)\in S\widetilde{\M}.\end{equation}
Consider the function $$v(x,\xi)=\displaystyle\int_{\tau_-(x,\xi)}^0\!\!\! \,\,\!\!\! a\bigg(\gamma_{x,\xi}(-s+\tau_-(x,\xi)),\dot{\gamma}_{x,\xi}(-s+
\tau_-(x,\xi))\bigg) \;ds,\; (x,\xi)\in S\M.$$
We have $v \in L^{\infty}(S\M)$ and verifies the equation \begin{equation*}\label{D v}
\D v(x,\xi)=a(x,\xi), \;\mbox{ in } S\M.
\end{equation*}
Indeed, we have
 \begin{align*}
\D v(x,\xi)&=\quad \p_r \bigg[\displaystyle \int_{\tau_-(\gamma_{x,\xi}(r))}^{0}a(\gamma_{x,\xi}(r-s+\tau_-
(x,\xi)),\dot{\gamma}_{x,\xi}(r-s+\tau_-(x,\xi))) \;ds \bigg]_{|r=0}\cr
&=\quad \p_r \bigg[ \displaystyle \int_{\tau_-(x,\xi)-r}^{0} a\big(\gamma_{x,\xi}(r-s+\tau_-(x,\xi)),\dot{\gamma}_{x,\xi}(r-s+\tau_
-(x,\xi)
)\big) \;ds \bigg]_{|r=0}\cr
&=\quad  a(x,\xi).
\end{align*}
Moreover, by definition of $\tau_-(x,\xi)$, we have $v(x,\xi)=0$ on $\p_-S\M$. Let us compute $v(x,\xi)$ on $\p_+S\M$. For $(x,\xi)\in \p_+S\M$, we set $$(x',\xi')=\phi_{\tau_-(x,\xi)}(x,\xi) \in \p_-S\M.$$ Then $\tau_-(x,\xi)=-\tau_+(x',\xi')$ and
\begin{align*}
 v(x,\xi)&=\quad \displaystyle \int_{\tau_-(x,\xi)}^{0}a\big(\gamma_{x,\xi}(-s+\tau_-
 (x,\xi)),
\dot{\gamma}_{x,\xi}(-s+\tau_-(x,\xi))\big) \;ds \cr
&=\quad \displaystyle \int_{-\tau_+(x',\xi')}^{0} a\big(\gamma_{x',\xi'}(-s),\dot{\gamma}_{x',\xi'}
(-s)\big) \;ds.
%\cr&=\quad \displaystyle \int^{\tau_+(x',\xi')}_{0} a\big(\gamma_{x',\xi'}(s),\dot{\gamma}_{x',\xi'}
%(s)\big) \;ds  .
\end{align*}
By a variable change, we get
\begin{equation*}
  v(x,\xi)=\displaystyle \int^{\tau_+(x',\xi')}_{0} a\big(\gamma_{x',\xi'}(s),\dot{\gamma}_{x',\xi'}
%(s)\big) \;ds
\end{equation*}
Using (\ref{int -TT=0}) and the fact that $T > Diam_g ~\M$ and that $a$ vanishes outside $\M$, we obtain
\begin{align*} v(x,\xi)&=\quad - \displaystyle \int_{-\tau_+(x',\xi')}^{0} a\big(\gamma_{x',\xi'}(s),\dot{\gamma}_{x',\xi'}
(s)\big) \;ds.
\end{align*}
But for $(x',\xi')\in \p_-S\M$ we have $\gamma_{x',\xi'}(s)\notin \M$ for $s \leq 0$.
We infer that
$ v(x,\xi)=0$ on $\p_+ S\M$ and  we conclude by having
\begin{equation*}
\left\{\begin{array}{lll}
\D v(x,\xi)&=a(x,\xi) &\mbox{  on } S \M  \cr
v(x,\xi)&=0 &\mbox{  on }\p S\M.
\end{array}
\right.
\end{equation*}
 Take $v$  solution of that system and set $l(x,\xi)=e^{v(x,\xi)}$. Then $l>0$ and $l, \D l \in L^{\infty}(S\M)$ with  $l=1$ on $\p S\M$ and we have $$a_1(x,\xi)=a_2(x,\xi)+\D(\log l(x,\xi)).$$ This completes the proof of Corollary \ref{corol0}.
\subsection{Stable determination of an isotropic absorption coefficient}
Our aim here is to prove the Theorem \ref{th1} where we have assumed that $a(x)$ does not depend on the direction $\theta$. We will prove that the albedo operator determines stably the absorption coefficient when this last is isotropic.
Denote by $\mathcal{X}$ the geodesic ray transform  on the simple manifold $\M$ defined as the linear operator
\begin{equation*}\label{X domaine}
\mathcal{X}:\mathcal{C}^\infty(\M)\To \mathcal{C}^\infty(\p_-S\M),
\end{equation*}
with
\begin{equation*}\label{X expression}
\mathcal{X} f(x,\theta)=\int_0^{\tau_+(x,\theta)}f(\gamma_{x,\theta}(t))\, \dd t,\quad (x,\theta)\in\p_-S\M.
\end{equation*}
$\mathcal{X} f(x,\theta)$ is a smooth function on $\p_-S\M$ because the integration limit $\tau_+(x,\theta)$ is a smooth function on $S \M \setminus S(\p \M)$.
The geodesic ray transform on  $\M$ can be extended to a bounded operator
\begin{equation}\label{X extension}
\mathcal{X}:H^k(\M)\To H^k(\p_-S\M),
\end{equation}
for every integer $k\geq 0$, see Theorem 4.2.1 of (\cite{[Sh1]}).\\
The Theorem \ref{th0} will allow us to stably reconstruct the geodesic ray transform of the absorption coefficient $a(x)$ from the albedo operator $\mathcal{A}_{a,k}$.
\begin{lemma}\label{lem3.4}
For admissible pairs $(a_1,k_1)$ and $(a_2,k_2)$ satisfying the conditions (\ref{cond th1}), there exist $C>0,$ such that
\begin{equation*}
\norm{\X(a)}_{L^2(\p_-S\M)}\leq C \norm{\mathcal{A}_{a_2,k_2}-\mathcal{A}_{a_1,k_1}} _1.
\end{equation*}
Here $C$ depends only on $\M$, $T$ and $C_0$.
\end{lemma}
\begin{proof}
Take $(x',\xi')\in \p_-S\M$ and $s_0 > 0$ such that $(x,\xi)=(\gamma_{x',\xi'}(s_0), \dot{\gamma}_{x',\xi'}(s_0))\in S\M$. We have $$ \mathcal{X}(a)(x',\xi')=\int_{0}^{\tau_+(x',\xi')}\!\!\! \,\,\!\!\! a(\gamma_{x',\xi'}(s)) \;ds=\int_{-s_0}^{\tau_+(x',\xi')-s_0}\!\!\! \,\,\!\!\! a(\gamma_{x,\xi}(s)) \;ds.$$ Since we have $\textrm{supp}(a) \subset \M$ and $T > Diam_{\g}\M +2 \epsilon$ (condition (\ref{condition T})), we get
\begin{equation*}\label{X -T T} \X(a)(x',\xi')=\int_{-T}^{T}\!\!\! \,\,\!\!\! a(\gamma_{x,\xi}(s)) \;ds.
\end{equation*}
Integrating  over $\p_-S\M$ and applying  Theorem \ref{th0} to $a(x)$ instead of $a(x,\theta)$, we obtain \begin{equation}\label{ray a albedo}
 \norm{\X(a)}_{L^2(\p_-S\M)}\leq C \norm{\mathcal{A}_{a_2,k_2}-\mathcal{A}_{a_1,k_1}} _1.
\end{equation}
\end{proof}
Now, to recover the absorption coefficient from its geodesic ray transform, we need the following well-known theorem which proof can be found in (\cite{[Sh1]} Th 4.3.3 p.119 ),  or in (\cite{[BZ]} Th 2.1):\\
%%%%%%%%%%%%%%%%%%%%%%%%%
For every simple compact Riemannian manifold $(\M,\g)$ with smooth boundary, if we assume $k^+(\M,\g)<1$, then there exists $C>0$ such that the following stability estimate
\begin{equation}\label{ray inverse}
 \Vert f \Vert _{L^2(\M)}\leq C \Vert \mathcal{X} f \Vert_{H^1(\p_-S\M)},
\end{equation}
holds for any $f\in H^1(\M)$.
%%%%%%%%%%%%%%%%%%%%%%%%%
%%%%%%%%%%%%%%%%%%%%%%%%%%%%%%%%%%%%%%%%%%%%%%%%%%%%%%%%%%%%%%%
Applying this Theorem, we get
\begin{equation}\label{estimate geodesical ray albedo}
\|a\|_{L^2(\M)}\leq C \|\X(a)\|_{H^1(\p_-S\M)}.
\end{equation}
Since we take $a_j\in \mathscr{A}(C_0,\eta)$ (\ref{A}), then $a\in H^2(\M)$. By  interpolation inequality, we get
\begin{align*}
\|\X(a)\|_{H^1(\p_-S\M)} & \leq C \|\X(a)\|^{\frac{1}{2}}_{L^2(\p_-S\M)}\|\X(a)\|^
{\frac{1}{2}}_{H^2(\p_-S\M)}.
\end{align*}
 Moreover, the operator $\X$ is bounded (\ref{X extension}), so we arrive to
$$\|a\|_{L^2(\M)}  \leq C \|\X(a)\|^{\frac{1}{2}}_{L^2(\p_-S\M)}\|a\|^
{\frac{1}{2}}_{H^2(\M)} \leq C\|\X(a)\|^{\frac{1}{2}}_{L^2(\p_-S\M)}.
$$
 Using the estimate (\ref{ray a albedo}), we obtain
 \begin{equation*}
   \|a\|_{L^2(\M)}  \leq C \norm{\mathcal{A}_{a_2,k_2}-\mathcal{A}_{a_1,k_1}}^{\frac{1}{2}} _1,
 \end{equation*}
where the constant $C$ depends on $\M$, $T$ and $C_0$.
Take $\eta' \in (1+\frac{n}{2},\eta)$. By Sobolev embedding and interpolation inequality, there exists $\delta\in(0,1)$ such that \begin{equation*}
\|a\|_{C^0(\M)}\leq C \|a\|_{H^{\eta'}(\M)}\leq C \|a\|^\delta_{L^2(\M)} \|a\|^{1-\delta}_{H^{\eta}(\M)}.
\end{equation*}
This yields to the stability estimate
\begin{equation}\label{norm a c^0}
\|a\|_{C^0(\M)}\leq C \norm{\mathcal{A}_{a_2,k_2}-\mathcal{A}_{a_1,k_1}}_1^{\frac{\delta}{2}},
\, \delta\in (0,1).\end{equation} Thus the proof of the Theorem \ref{th1} is complete and we derive obviously \ref{corol}.
%%%%%%%%%%%%%%%%%%%%%%%%%%%%%%%%%%%%%%%%%%%%%%%
\section{A uniqueness result for the scattering coefficient}
\setcounter{equation}{0}
%%%%%%%%%%%%%%%%%%%%%%%%%%%%
In this section, we prove Theorem \ref{th2}, an identification result recovering the spatial part of the scattering coefficient $k$ from the Albedo operator. We consider $a_j,\, j=1,2$, depending on both the position $x$ and the direction $\theta$.
% Like in the previous section,  we let $a_1,\,a_2\in\mathscr{A}(C_0,\eta)$ and $k_1,k_2\in\mathscr{K}(C_1,C_2)$ such that  $\p_x^\alpha a_1=\p_x^\alpha a_2$, $x\in \p \M,\; \alpha \in \N^n,\; | \alpha| \leq 2$. Furthermore, we assume that  for any $x$ near  the boundary $\p\M$, we have $k_1(x,\theta,\theta')=k_2(x,\theta,\theta')$, $\forall (\theta,\theta')\in S_x\M \times S_x\M$ . We set
%$$a=a_1-a_2 ,\quad k=k_1-k_2.$$ Recall that we have extended the metric $\g$ by the identity outside $\M$. We also extend $a_{1},a_{2},k_1,k_2$ in such a way that $a(x)=0$ and $k(x,\theta,\theta')=0$ for all $x\in \widetilde{\M}\setminus \M$ and $(\theta,\theta')\in S_x\widetilde{\M} \times S_x\widetilde{\M}$.\\ For $j=1,2$, we take \begin{equation*} \phi_j \in C_0^{\infty}(\widetilde{\M}\setminus \M,C(\mathbb{S}^{n-1}))
%\end{equation*} that we extend to $S\R^n$ by zero. We set
%\begin{equation*}
%\alpha_j(t,x,\theta)=\phi_j(\gamma_{x,\theta}(-t),\dot{\gamma}_
%{x,\theta}(-t)),\;\; \forall(t,x,\theta)\in \R \times S\M
%\end{equation*} as defined in (\ref{alpha}). Next, we denote by $\beta_{a_1}$ and $\beta_{-a_2}$  the solutions of (\ref{eq beta}) defined in (\ref{beta}) by taking respectively $a=a_1$ and $a=-a_2$. We set
%\begin{equation*}
%\beta_a(t,x,\theta)=\beta_{a_1}(t,x,\theta)\beta_{-a_2}(t,x,\theta).
%\end{equation*}
We choose $\rho_j\in C_0^{\infty}(\widetilde{\M}\setminus \M)$, for $j=1,2$ and $\Psi_h,\; \Psi_{h'}$ as defined in (\ref{psi h}), for $h,h'\in (0,1)$.  Pick up $x_0 \in \widetilde{\M}\setminus \M$ and $\xi \in \mathbb{S}^{n-1}$. We let $x\in \widetilde{\M}$ and we denote by  $\xi_x=\mathcal{P}(\xi;x_0 ,x)$.
We define the functions $ \phi_1$ and $ \phi_2$ by
\begin{equation}\label{rhoj phi h h'}
   \phi_1(x,\theta)=\rho_1(x)\Psi_h(\xi_x,\theta) \mbox{ and  } \phi_2(x,\theta)=\rho_2(x)\Psi_{h'}(\xi_x,\theta).
   \end{equation}
% They satisfy the transport equation (\ref{eq alpha}).
 Let $p,p' \geq 1$  such that $\frac{1}{p}+\frac{1}{p'}=1$. Applying Lemma \ref{uj estimate rj phij}, we find geometric optics solutions $u_1 \in C^1([0,T];L^p(S\M))\cap  C^0([0,T];\mathcal{W}^p)$  and $u_2 \in C^1([0,T];L^{p'}(S\M))\cap  C^0([0,T];\mathcal{W}^{p'})$ to the respective Boltzmann equations (\ref{u1 u(0)}) and (\ref{syst u2}), given by \begin{equation*}
  u_1(t,x,\theta)=(\alpha_1 \beta_{a_1})(t,x,\theta) e^{i\lambda(t+\tau_-(x,\theta))}+r^1_{h,\lambda} (t,x,\theta),
  \end{equation*}
and \begin{equation*}
  u_2(t,x,\theta)=(\alpha_2 \beta_{-a_2})(t,x,\theta) e^{-i\lambda(t+\tau_-(x,\theta))}+r^2_{h',\lambda} (t,x,\theta),
\end{equation*}
where $r^1_{h,\lambda}$ and $r^2_{h',\lambda}$ satisfy respectively the initial and final BVPs (\ref{r1 lamda}) for $a=a_1$ and (\ref{r2 lamda}) for $a=a_2$ and the conditions (\ref{norm r1 L2}) and (\ref{syst r2}).
%In this section, we assume that $\mathcal{A}_{a_2,k_2}=\mathcal{A}_{a_1,k_1}$ and so $a_1=a_2$ in $\M$.
We assume that $$\mathcal{A}_{a_2,k_2}=\mathcal{A}_{a_1,k_1}.$$ Upon using the corollary \ref{corol0}, there exist a positive function $l, \D l \in L^{\infty}(S\M)$ with  $l=1$ on $\p S\M$ such that $$a_1(x,\xi)=a_2(x,\xi)+\D(\log l(x,\xi)).$$ Set $$a'_2(x,\xi)=a_2(x,\xi)+\D(\log l(x,\xi)),\mbox{  and  }\; k'_2(x, \theta,\theta') = k_2(x,\theta,\theta')\frac{l(x,\theta)}{l(x,\theta')},$$
and $$a'=a_1-a'_2,\;k'=k_1-k'_2 ,$$ then we have $\mathcal{A}_{a'_2,k'_2}=\mathcal{A}_{a_2,k_2}$. Hence  $$\mathcal{A}_{a'_2,k'_2}-\mathcal{A}_{a_1,k_1}=0\,\mbox{ and } a'=0.$$
Substituting  $a_2$ by $a'_2$ and  $a$ by $a'$ in the identity (\ref{equation integr1}), we get
\begin{equation}\label{identity 4.2}
\int_0^T\!\!\!\int_{S\M} u_2 \I_{k'}[ u_1](t,x,\theta) \, \dvv \,dt = 0.
\end{equation}
%We first estimate the right-hand side of (\ref{identity 4.2}).
Replacing $u_1$ and $u_2$ by their expressions, we obtain that
\begin{equation}\label{J}
\int_0^T\!\!\!\int_{S\M} u_2 \I_{k'}[ u_1](t,x,\theta) \, \dvv \,dt= J_{h,h',\lambda}(\xi)+ J^1_{h,h',\lambda}(\xi)+  J^2_{h,h',\lambda}(\xi)+J^3_{h,h',\lambda}(\xi),
\end{equation}
where $$
J_{h,h',\lambda}(\xi):=\int_0^T\!\!\!\int_{S\M} \alpha_2\beta_{-a'_2}\psi_{-\lambda} \I_{k'}[ \alpha_1\beta_{a_1}\psi_{\lambda}](t,x,\theta) \, \dvv \,dt ,
$$
  $$ J^1_{h,h',\lambda}(\xi):=\int_0^T\!\!\!\int_{S\M} \alpha_2\beta_{-a'_2}\psi_{-\lambda} \I_{k'}[ r^1_{h,\lambda}](t,x,\theta) \, \dvv \,dt,
   $$
   $$ J^2_{h,h',\lambda}(\xi):= \int_0^T\!\!\!\int_{S\M} r^2_{h',\lambda} \I_{k'}[ \alpha_1\beta_{a_1}\psi_{\lambda}](t,x,\theta) \, \dvv \,dt,
     $$
      $$ J^3_{h,h',\lambda}(\xi):=  \int_0^T\!\!\!\int_{S\M} r^2_{h',\lambda} \I_{k'}[ r^1_{h,\lambda}](t,x,\theta) \, \dvv \,dt.$$
We aim to take the limit of these terms when $h,h' \rightarrow 1$ and $\lambda \rightarrow \infty$.
We have a preliminary result.
\begin{lemma}\label{lem4.1}
  For any $\xi \in \s^{n-1}$ and any $\rho_j\in C_0^{\infty}(\widetilde{\M}\setminus \M)$,  $j=1,2$, we have
   \begin{equation*}
 \displaystyle \lim_{h,h' \rightarrow 1} J_{h,h',\lambda}(\xi)= \int_0^T\!\!\!\int_{\widetilde{\M} \setminus \M}\!\!\! (\rho_1 \rho_2)(x) k(\gamma_{x,\xi}(t),\dot{\gamma}
_{x,\xi}(t),\dot{\gamma}
_{x,\xi}(t)) \, \dv\,dt.\end{equation*}
\end{lemma}
\begin{proof}
We first note that using (\ref{prop1}) and (\ref{prop2}), we have
\begin{equation}\label{domine}
\big| \alpha_2\beta_{-a'_2}\psi_{-\lambda} \I_{k'}[ \alpha_1\beta_{a_1}\psi_{\lambda}](t,x,\theta)  \big| \leq C(1-h')^{1-n},
\end{equation} where $C$ depends on $n,T,C_0,C_1$ and $C_2$. We aim to use the Lebesgue's Theorem in order to take the limit of $J_{h,h',\lambda}$ when $h\rightarrow1$. Recall that we have fixed $x_0\in \widetilde{\M}\setminus \M$ and $\xi \in \mathbb{S}^{n-1}$ and we have denoted $\xi_x=\mathcal{P}(\xi;x_0,x)$. Since $k'(x,\theta,\theta')$ vanishes for $x$ outside $\M$, we have
 \begin{align*}
J_{h,h',\lambda}(\xi)= & \int_0^T\!\!\!\int_{\widetilde{\M}}\!\!\!\int_{S_x\widetilde{\M}} \rho_2(\gamma_{x,\theta}(-t))\Psi_{h'}(\xi_{\gamma_{x,\theta}(-t)},
\dot{\gamma}_{x,\theta}(-t))
(\beta_{-a'_2}\psi_{-\lambda})(t,x,\theta)  \!\!\!\int_{S_x\widetilde{\M}}k'(x,\theta,\theta')\cr &
\rho_1(\gamma_{x,\theta'}(-t)) \Psi_{h}(\xi_{\gamma_{x,\theta'}(-t)},\dot{\gamma}_{x,\theta'}(-t)) (\beta_{a_1}\psi_{\lambda})(t,x,\theta')  \dd\omega_x(\theta') \, \dd\omega_x(\theta) \, \dv(x)\,dt .\end{align*}
We perform the change of variables from $(-t,x,\theta')$ to $(y,\widetilde{\theta'})$ where $$y=\gamma_{x,\theta'}(-t),\;\widetilde{\theta'}= \dot{\gamma}_{x,\theta'}(-t)\; \mbox{ and } \theta_1=\mathcal{P}(\theta;x,y).$$ So $x=\gamma_{y,\widetilde{\theta'}}(t)$,  $\theta'=\dot{\gamma}_{y,\widetilde{\theta'}}(t)$ and $\theta=\mathcal{P}(\theta_1;y,\gamma_{y,\widetilde{\theta'}}(t))$.
We set $\xi_1=\mathcal{P}(\xi_x;x,y)$ and $\widetilde{\xi}_1=\mathcal{P}(\xi_x;x,\gamma_{x,\theta}(-t))$.
 Thus,
 \begin{align*}
J_{h,h',\lambda}(\xi)= & \int_0^T\!\!\!\int_{\widetilde{\M}}\!\!\!\int_{S_y\widetilde{\M}} \!\!\!\int_{S_y\widetilde{\M}} \rho_2(\gamma_{\gamma_{y,\widetilde{\theta'}}(t),\mathcal{P}
(\theta_1;y,\gamma_{y,\widetilde{\theta'}}(t))}(-t))
\Psi_{h'}(\widetilde{\xi}_1,\dot{\gamma}_{\gamma_{y,\widetilde{\theta'}}(t),
\mathcal{P}
(\theta_1;y,\gamma_{y,\widetilde{\theta'}}(t))}(-t))
 \cr & (\beta_{-a'_2}\psi_{-\lambda})(t,\gamma_{y,\widetilde{\theta'}}(t),
 \mathcal{P}(\theta_1;y,\gamma_{y,\widetilde{\theta'}}(t))) k'\big(\gamma_{y,\widetilde{\theta'}}(t),\mathcal{P}(\theta_1;y,\gamma_
{y,\widetilde{\theta'}}(t)),
\dot{\gamma}_{y,\widetilde
{\theta'}}(t)\big)\cr &   \rho_1(y)\Psi_{h}(\xi_1,\widetilde{\theta'})  (\beta_{a_1} \psi_{\lambda})(t,\gamma_{y,\widetilde{\theta'}}(t),\dot{\gamma}
_{y,\widetilde{\theta'}}(t))  \dd\omega_y(\tilde{\theta'}) \, \dd\omega_y(\theta_1) \, \dv(y)\,dt .\end{align*}
 Thanks to (\ref{domine}), the Lebesgue's theorem and the Poisson formula  (\ref{prop4}) lead to
 \begin{align}\label{limh1}
\displaystyle \lim_{h \rightarrow 1} J_{h,h',\lambda}(\xi)=  & \int_0^T\!\!\!\int_{\widetilde{\M}}\!\!\!\int_{S_y\widetilde{\M}} \!\!\! \rho_2(\gamma_{\gamma_{y,\xi_1}(t),
\mathcal{P}(\theta_1;y,\gamma_{y,\xi_1}(t))}(-t))
\Psi_{h'}(\widetilde{\xi}_2,\dot{\gamma}_{\gamma_{y,\xi_1}(t),\mathcal{P}
(\theta_1;y,\gamma_{y,\xi_1}(t))}(-t)) \cr &
(\beta_{-a'_2}\psi_{-\lambda})(t,\gamma_{y,\xi_1}(t),\mathcal{P}(\theta_1;y,\gamma
_{y,\xi_1}(t))) k'(\gamma_{y,\xi_1}(t),\mathcal{P}(\theta_1;y,\gamma
_{y,\xi_1}(t)),\dot{\gamma}_{y,\xi_1}(t))  \cr &  \rho_1(y)(\beta_{a_1}
\psi_{\lambda})(t,\gamma_{y,\xi_1}(t),\dot{\gamma}_{y,\xi_1}(t))  \, \dd\omega_y(\theta_1) \, \dv(y)\,dt .\end{align}
Here, $\widetilde{\xi}_2=\mathcal{P}(\dot{\gamma}_{y,\xi_1}(t);
\gamma_{y,\xi_1}(t), \gamma_{\gamma_{y,\xi_1}(t),
\mathcal{P}(\theta_1;y,\gamma_{y,\xi_1}(t))}(-t))$.
To simplify the expression, we make another  variables change by setting $$(z,\xi_2)=(\gamma_{y,\xi_1}(t),\dot{\gamma}_{y,\xi_1}(t))\;\mbox{ and }\; \theta_2=\mathcal{P}(\theta_1;y,z).$$
So, we have $y=\gamma_{z,\xi_2}(-t)$, $\xi_1=\dot{\gamma}_{z,\xi_2}(-t)$ and $\theta_1=\mathcal{P}(\theta_2;z,\gamma_{z,\xi_2}(-t))$. Then, $\widetilde{\xi}_2=\mathcal{P}(\xi_2;z,\gamma_{z,
\theta_2}(-t))$
Therefore, (\ref{limh1}) reads
 \begin{align*}
\displaystyle \lim_{h \rightarrow 1} J_{h,h',\lambda}(\xi)=  & \int_0^T\!\!\!\int_{\widetilde{\M}}\!\!\!\int_{S_z\widetilde{\M}} \!\!\! \rho_2(\gamma_{z,
\theta_2}(-t))\Psi_{h'}(\widetilde{\xi}_2,\dot{\gamma}_{z,\theta_2}(-t))
(\beta_{-a'_2}\psi_{-\lambda})(t,z,\theta_2)  \cr & k'(z,\theta_2,\xi_2) \rho_1(\gamma_{z,\xi_2}(-t))(\beta_{a_1}
\psi_{\lambda})(t,z,\xi_2)  \, \dd\omega_z(\theta_2) \, \dv(z)\,dt .\end{align*}
By changing $(z,\theta_2)$ into $$(\widetilde{z},\widetilde{\theta}_2)=(\gamma_{z,
\theta_2}(-t),\dot{\gamma}_{z,\theta_2}(-t)),$$ we get $\widetilde{\xi}_2=\mathcal{P}(\xi_2;z,\widetilde{z})$ and then
 \begin{align*}
\displaystyle \lim_{h \rightarrow 1} J_{h,h',\lambda}(\xi)= & \int_0^T\!\!\!\int_{\widetilde{\M}}\!\!\!\int_{S_{\widetilde{z}}\widetilde{\M}} \!\!\! \rho_2(\widetilde{z})\Psi_{h'}(\widetilde{\xi}_2,\widetilde{\theta}_2)
(\beta_{-a'_2}\psi_{-\lambda})(t,\gamma_{\widetilde{z},
\widetilde{\theta}_2}(t),
\dot{\gamma}_{\widetilde{z},
\widetilde{\theta}_2}(t))  \cr  & k'(\gamma_{\widetilde{z},\widetilde{\theta}_2}(t),\dot{\gamma}
_{\widetilde{z},\widetilde{\theta}_2}(t),\mathcal{P}(\widetilde{\xi}_2;
\widetilde{z},\gamma_{\widetilde{z},\widetilde{\theta}_2}(t)))
 \rho_1(\gamma_{\gamma_{\widetilde{z},\widetilde{\theta}_2}(t),
\mathcal{P}(\widetilde{\xi}_2;\widetilde{z},\gamma_{\widetilde{z},
\widetilde{\theta}_2}(t))}(-t))\cr &
 (\beta_{a_1}
\psi_{\lambda})(t,\gamma_{\widetilde{z},
\widetilde{\theta}_2}(t),\mathcal{P}(\widetilde{\xi}_2;
\widetilde{z},\gamma_{\widetilde{z},\widetilde{\theta}_2}(t)))  \, \dd\omega_{\widetilde{z}}(\widetilde{\theta}_2) \, \dv(\widetilde{z})\,dt .\end{align*}
Using (\ref{prop4}) and the Lebesgue's Theorem, we obtain
  \begin{align*}
\displaystyle \lim_{h' \rightarrow 1} \displaystyle \lim_{h \rightarrow 1} J_{h,h',\lambda}(\xi)= & \int_0^T\!\!\!\int_{\widetilde{\M}}\!\!\! \rho_2(\widetilde{z})(\beta_{-a'_2}\psi_{-\lambda})(t,\gamma_
{\widetilde{z},\widetilde{\xi}_2}(t),
\dot{\gamma}_{\widetilde{z},\widetilde{\xi}_2}(t)) %k(\gamma_{\widetilde{z},\widetilde{\xi}_2}(t),\dot{\gamma}
%_{\widetilde{z},\widetilde{\xi}_2}(t),\dot{\gamma}
%_{\widetilde{z},\widetilde{\xi}_2}(t))
 \rho_1(\widetilde{z})\cr  & (\beta_{a_1}
\psi_{\lambda})
(t,\gamma_{\widetilde{z},\widetilde{\xi}_2}(t),
\dot{\gamma}_{\widetilde{z},\widetilde{\xi}_2}(t)) k'(\gamma_{\widetilde{z},
\widetilde{\xi}_2}(t),\dot{\gamma}
_{\widetilde{z},\widetilde{\xi}_2}(t),\dot{\gamma}
_{\widetilde{z},\widetilde{\xi}_2}(t)) \,\dv(\widetilde{z})\,dt
\cr = & \int_0^T\!\!\!\int_{\widetilde{\M} \setminus \M}\!\!\! (\rho_1\rho_2)(\widetilde{z})
k'(\gamma_{\widetilde{z},
\widetilde{\xi}_2}(t),\dot{\gamma}
_{\widetilde{z},\widetilde{\xi}_2}(t),\dot{\gamma}
_{\widetilde{z},\widetilde{\xi}_2}(t)) \, \dv\,dt.
\end{align*} In the last equality we have used the fact that $\mathrm{supp}(\rho_1\rho_2) \subset \widetilde{\M}\setminus \M$ and that $a'_2=a_1$. Note that $\xi$ is arbitrary in $\mathbb{S}^{n-1}$ and that $\widetilde{\xi}_2$ is obtained from $\xi$ by parallel transport. Then relabeling $\widetilde{z}$ as $x$ and $\widetilde{\xi}_2$ as $\xi$, we end up by getting
 $$\displaystyle \lim_{h' \rightarrow 1} \displaystyle \lim_{h \rightarrow 1} J_{h,h',\lambda}(\xi)=\int_0^T\!\!\!\int_{\widetilde{\M} \setminus \M}\!\!\! (\rho_1 \rho_2)(x)k'(\gamma_{x,\xi}(t),
 \dot{\gamma}
_{x,\xi}(t),\dot{\gamma}
_{x,\xi}(t)) \, \dv\,dt.$$ Notice that for $\theta=\theta'$, we have  $k'_2(x,\theta,\theta)=k_2(x,\theta,\theta)$, this completes the proof of the Lemma.
\end{proof}
Now, to compute the limit as $h',h \rightarrow 1$ of $J^j_{h,h',\lambda},\;j\in\{1,2,3\}$, defined in (\ref{J}), we need the following Lemma.
\begin{lemma}\label{lem4.2}
  Let $x_0 \in \widetilde{\M} \setminus \M,\; \xi \in \mathbb{S}^{n-1}$ and $r^1_{h,\lambda}$ the unique solution of the IVBP (\ref{r1 lamda}) by taking $a=a_1$ and $k=k_1$. We denote by $R^1_{\lambda,\xi}$ the unique solution of
  \begin{equation*}
 \left\{
  \begin{array}{lll}
  \partial_t R^1_{\lambda,\xi}+\D R^1_{\lambda,\xi}+a_1(x)R^1_{\lambda,\xi}=
\I_{k_1}[R^1_{\lambda,\xi}](t,x,\theta)+k_1(x,\theta,\xi_x)\rho_1(\gamma_{x,\xi_x}(-t))( \beta_{a_1} \psi_{\lambda})(t,x,\xi_x)  & \textrm{in }\; Q_T,\cr
R^1_{\lambda,\xi}(0, x,\theta )=0 & \textrm{in }\; S\M ,\cr
R^1_{\lambda,\xi}(t,x,\theta)=0 & \textrm{on } \; \Sigma_-.
\end{array}
\right.
\end{equation*}
%Here $\xi_x=\mathcal{P}(\xi;x_0,x)$.
Then, we have $$\displaystyle \lim_{h \rightarrow 1} \|r^1_{h,\lambda} - R^1_{\lambda,\xi} \|_{L^2(Q_T)} =0.$$
Therefore,  $$\displaystyle \lim_{h \rightarrow 1} \|\I_{k_1}[r^1_{h,\lambda}]-\I_{k_1}[R^1_{\lambda,\xi}]\|_{L^2(Q_T)} =0.$$
\end{lemma}
\begin{proof} Being solution of  (\ref{r1 lamda}), $r^1_{h,\lambda}$ satisfies the equation $$ \partial_t r^1_{h,\lambda}+\D r^1_{h,\lambda}+a_1(x)r^1_{h,\lambda}=
\I_{k_1}[r^1_{h,\lambda}](t,x,\theta)+\I_{k_1}[ \alpha_1 \beta_{a_1} \psi_{\lambda}](t,x,\xi_x).$$
 Let $v_{h,\lambda}(t,x,\theta)=r^1_{h,\lambda} - R^1_{\lambda,\xi}$, then we have
  \begin{equation*}
 \left\{
  \begin{array}{lll}
  \partial_t v_{h,\lambda}+\D v_{h,\lambda}+a_1(x)v_{h,\lambda}=
\I_{k_1}[v_{h,\lambda}]+\I_{k_1}[ \alpha_1 \beta_{a_1} \psi_{\lambda}]-\mathcal{L}_{\lambda,\xi}  & \textrm{in }\; Q_T,\cr
v_{h,\lambda}(0, x,\theta )=0 & \textrm{in }\; S\M ,\cr
v_{h,\lambda}(t,x,\theta)=0 & \textrm{on } \; \Sigma_-.
\end{array}
\right.
\end{equation*}
where $$\mathcal{L}_{\lambda,\xi} (t,x,\theta)=k_1(x,\theta, \xi_x)\rho_1(\gamma_{x,\xi_x}(-t))( \beta_{a_1} \psi_{\lambda})(t,x,\xi_x).$$
Therefore the Lemma \ref{lemma2} gives
$$\|v_{h,\lambda} \|_{L^2(Q_T)} \leq C \|\I_{k_1}[\alpha_1 \beta_{a_1} \psi_{\lambda}]-\mathcal{L}_{\lambda,\xi}\|_{L^2(Q_T)},$$
where $C$ depends only on $\M$ and $T$.
To prove the Lemma, we should prove that
$$\displaystyle \lim_{h  \rightarrow 1} \|\I_{k_1}[\alpha_1 \beta_{a_1} \psi_{\lambda}]
-\mathcal{L}_{\lambda,\xi}\|_{L^2(Q_T)}=0. $$
Since $\mathrm{supp}(k_1)\subset \M,$ we have
\begin{align*}
\int_0^T\!\!\!\int_{\M}\!\!\!\int_{S_x\M} \!\!\! \I_{k_1}[\alpha_1 \beta_{a_1} \psi_{\lambda}](t,x,\theta)&\,\dd\omega_x(\theta)  \, \dv(x)\,dt  =  \int_0^T\!\!\!\int_{\widetilde{\M}}\!\!\!\int_{S_x\widetilde{\M}} \!\!\!\int_{S_x\widetilde{\M}} \!\!\! k_1(x,\theta,\theta') \rho_1(\gamma_{x,\theta'}(-t))\cr &\Psi_{h}(\xi_{\gamma_{x,\theta'}(-t)},
\dot{\gamma}_{x,\theta'}(-t))(\beta_{a_1}
\psi_{\lambda})(t,x,\theta')  \, \dd\omega_x(\theta')  \,\dd\omega_x(\theta)  \, \dv(x)\,dt .\end{align*}
By changing $(x,\theta')$ into $$(\widetilde{x},\widetilde{\theta'})=(\gamma_{x,
\theta'}(-t),\dot{\gamma}_{x,\theta'}(-t)),$$ by setting  $\widetilde{\xi}=\mathcal{P}(\xi_x;x,\widetilde{x})$ and $\widetilde{\theta}=\mathcal{P}(\theta;x,\widetilde{x})$, we get
\begin{align*}
\int_0^T\!\!\!\int_{\widetilde{\M}}\!\!\!\int_{S_{x}\widetilde{\M}} \!\!\! \I_{k_1}[\alpha_1 \beta_{a_1} \psi_{\lambda}] =  & \int_0^T\!\!\!\int_{\widetilde{\M}}\!\!\!\int_{S_{\tilde{x}}
\widetilde{\M}} \!\!\!\int_{S_{\tilde{x}}\widetilde{\M}} \!\!\! k_1(\gamma_{\tilde{x},\widetilde{\theta'}}(t),\mathcal{P}(\widetilde{\theta}
;\widetilde{x},\gamma_{\tilde{x},\widetilde{\theta'}}(t)),
\dot{\gamma}_{\tilde{x},\widetilde{\theta'}}(t))\rho_1(\tilde{x})
\Psi_{h}(\tilde{\xi},\widetilde{\theta'})\cr &(\beta_{a_1}
\psi_{\lambda})(t,\gamma_{\tilde{x},\widetilde{\theta'}}(t),
\dot{\gamma}_{\tilde{x},\widetilde{\theta'}}(t))  \, \dd\omega_{\tilde{x}}(\widetilde{\theta'})  \, \dd\omega_{\tilde{x}}(\widetilde{\theta})  \, \dv(\widetilde{x})\,dt .\end{align*}
 Applying the  formula (\ref{prop4}) we have
 \begin{align*}
\displaystyle \lim_{h  \rightarrow 1}\int_0^T\!\!\!\int_{\M}\!\!\!\int_{S_{x}\M} \!\!\! \I_{k_1}[\alpha_1 \beta_{a_1} \psi_{\lambda}] =  & \int_0^T\!\!\!\int_{\widetilde{\M}}\!\!\!\int_{S_{\tilde{x}}
\widetilde{\M}}  \!\!\! k_1(\gamma_{\tilde{x},\widetilde{\xi}}(t),\mathcal{P}(\widetilde
{\theta}
;\widetilde{x},\gamma_{\tilde{x},\widetilde{\xi}}(t)),
\dot{\gamma}_{\tilde{x},\widetilde{\xi}}(t))\rho_1(\tilde{x})
\cr &(\beta_{a_1}
\psi_{\lambda})(t,\gamma_{\tilde{x},\widetilde{\xi}}(t),
\dot{\gamma}_{\tilde{x},\widetilde{\xi}}(t))  \, \dd\omega_{\tilde{x}}(\widetilde{\theta})  \, \dv(\widetilde{x})\,dt .\end{align*}
Performing the variable change $x=\gamma_{\tilde{x},\widetilde{\xi}}(t)$ and $\theta=\mathcal{P}(\widetilde
{\theta}
;\widetilde{x},\gamma_{\tilde{x},\widetilde{\xi}}(t))$, and noticing that $\dot{\gamma}_{\tilde{x},\widetilde{\xi}}(t))=\xi_x$, we obtain
 \begin{align*}
\displaystyle \lim_{h  \rightarrow 1}\int_0^T\!\!\!\int_{\widetilde{\M}}\!\!\!\int_{S_{x}\widetilde{\M}} \!\!\! \I_{k_1}[\alpha_1 \beta_{a_1} \psi_{\lambda}] =  & \int_0^T\!\!\!\int_{\widetilde{\M}}\!\!\!\int_{S_{x}
\widetilde{\M}}  \!\!\! k_1(x,\theta,\xi_x),\rho_1(\gamma_{x,\xi_x}(-t))
\cr &(\beta_{a_1}
\psi_{\lambda})(t,x,\xi_x)  \, \dd\omega_{x}(\theta)  \, \dv(x)\,dt .\end{align*}
The limit is taken in the topology of $L^2(Q_T)$ and uniformly on $Q_T$.
  This implies that  $$\displaystyle \lim_{h  \rightarrow 1} \|\I_{k_1}[\alpha_1 \beta_{a_1} \psi_{\lambda}]-\mathcal{L}_{\lambda,\xi}\|_{L^2(Q_T)}=0. $$ We deduce that  $$\displaystyle \lim_{h \rightarrow 1} \|r^1_{h,\lambda} - R^1_{\lambda,\xi} \|_{L^2(Q_T)} =0 \mbox{ and }\displaystyle \lim_{h \rightarrow 1} \|\I_{k_1}[r^1_{h,\lambda}]-\I_{k_1}[R^1_{\lambda,\xi}]\|_{L^2(Q_T)} =0.$$
\end{proof}
Miming the proof of Lemma \ref{lem4.2}, we have the following result.
\begin{lemma}\label{4.3}
 Let $x_0 \in \widetilde{\M} \setminus \M,\; \xi \in \mathbb{S}^{n-1}$ and $r^2_{h',\lambda}$ the unique solution of the final VBP (\ref{r2 lamda}) for $a=a'_2$ and $k=k'_2$. We denote by $R^2_{\lambda,\xi}$ the unique solution of
  \begin{equation*}
  \left\{
  \begin{array}{lll}
\partial_t R^2_{\lambda,\xi}+\D R^2_{\lambda,\xi}-a'_2(x)R^2_{\lambda,\xi}=-
\I^*_{k'_2}[R^2_{\lambda,\xi}]-k'_2(x,\theta,\xi_x) \rho_2(\gamma_{x,\xi_x}(-t))( \beta_{a'_2} \psi_{-\lambda})(t,x,\xi_x)  & \textrm{in }\; Q_T,\cr
R^2_{\lambda,\xi}(T, x,\theta )=0 & \textrm{in }\; S\M ,\cr
R^2_{\lambda,\xi}(t,x,\theta)=0 & \textrm{on } \; \Sigma_-.
\end{array}
\right.
\end{equation*}
Then, we have $$\displaystyle \lim_{h' \rightarrow 1} \|r^2_{h',\lambda} - R^2_{\lambda,\xi} \|_{L^2(Q_T)} =0.$$
Moreover,  $$\displaystyle \lim_{h' \rightarrow 1} \|\I^*_{k'_2}[r^2_{h',\lambda}]-\I^*_{k'_2}[R^2_{\lambda,\xi}]\|_{L^2(Q_T)} =0.$$
\end{lemma}
\begin{lemma}\label{limJ lambda=0}
Set  $$J_{\lambda}(\xi)=\displaystyle \lim_{h',h \rightarrow 1}\big(J^1_{h,h',\lambda}(\xi)+J^2_{h,h',\lambda}(\xi)+
J^3_{h,h',\lambda}(\xi)\big).$$
Then we have $$\displaystyle \lim_{\lambda \rightarrow +\infty} J_{\lambda}(\xi)=0.$$
\end{lemma}
\begin{proof}
 We proceed like in Lemma \ref{lem4.1}, using Lemma \ref{lem4.2} and property (\ref{prop4}), we find
  \begin{align*}
  \displaystyle \lim_{h \rightarrow 1}J^1_{h,h',\lambda}(\xi)=\int_{Q_T} & \rho_2(\gamma_{x,\theta}(-t))\Psi_{h'}(\xi_{\gamma_{x,\theta}(-t)}
  ,\dot{\gamma}_{x,\theta}(t))( \beta_{-a'_2}\psi_{-\lambda})(t,x,\theta)\cr & \I_k'[R^1_{\lambda,\xi}](t,x,\theta) \, \dvv \,dt,
  \end{align*}
 $$\displaystyle \lim_{h \rightarrow 1}J^2_{h,h',\lambda}(\xi)= \int_{Q_T} r^2_{h',\lambda} k'(x,\theta,\xi_x) \rho_1(\gamma_{x,\xi_x}(-t))\beta_{a_1}\psi_{\lambda}](t,x,\xi_x) \, \dvv \,dt,$$
   $$\displaystyle \lim_{h \rightarrow 1} J^3_{h,h',\lambda}(\xi)=  \int_{Q_T} r^2_{h',\lambda}(t,x,\theta) \I_{k'}[ R^1_{\lambda,\xi}](t,x,\theta) \, \dvv \,dt.$$
 By the same arguments, we prove that taking the limit as $h' \rightarrow 1$, we find
 \begin{multline*}
 J_{\lambda}(\xi)=\displaystyle \lim_{h' \rightarrow 1} \displaystyle \lim_{h \rightarrow 1} \big(J^1_{h,h',\lambda}(\xi)+J^2_{h,h',\lambda}(\xi)+
J^3_{h,h',\lambda}(\xi)\big)=\cr
\int_0^T\!\!\!\int_{\M}\bigg(\rho_1(\gamma_{x,\xi_x}
(-t))(\beta_{-a'_2}\psi_{-\lambda})(t,x,\xi_x)(t)\I_{k'}[R^1_
{\lambda,\xi}] + \I^*_{k'}[R^2_{\lambda,\xi}]\rho_2(\gamma_{x,\xi_x}
(-t))(\beta_{a_1}\psi_{\lambda} \bigg)\, \dv \,dt \cr
+\int_{Q_T} R^2_{\lambda,\xi}(t,x,\theta) \I_{k'}[ R^1_{\lambda,\xi}](t,x,\theta) \, \dvv \,dt.
\end{multline*}
 We emphasize that the operator $\I_k$ is compact and that for $j=1,2$, $R^j_{\lambda,\xi} \rightharpoonup 0$ in $L^2(S\M)$, as $\lambda \rightarrow \infty$. This leads to
$$\displaystyle \lim_{\lambda \rightarrow \infty} \|\I_{k'}[R^1_{\lambda,\xi}](t,\cdot,\cdot)\|_{L^2(S\M)} = \lim_{\lambda \rightarrow \infty} \|\I^*_{k'}[R^2_{\lambda,\xi}](t,\cdot,\cdot)\|_{L^2(S\M)} =0,\; \forall t\in(0,T).$$ Applying the Lebesgue's Theorem, we obtain $$\displaystyle \lim_{\lambda \rightarrow \infty} \|\I_{k'}[R^1_{\lambda,\xi}]\|_{L^2(Q_T)} = \lim_{\lambda \rightarrow \infty} \|\I^*_{k'}[R^2_{\lambda,\xi}]\|_{L^2(Q_T)} =0.$$
This completes the proof of the Lemma.
%\end{multline*}
\end{proof}
Note that $ \displaystyle \lim_{h,h' \rightarrow 1} J_{h,h',\lambda}(\xi)$ in Lemma \ref{lem4.1} does not depend on $\lambda$. Using Lemma \ref{lem4.1}  and lemma \ref{limJ lambda=0}, the equality (\ref{identity 4.2}) yields to
\begin{align}\label{4.9}
  \int_0^T\!\!\!\int_{\widetilde{\M} \setminus \M}\!\!\! (\rho_1 \rho_2)(x)k(\gamma_{x,\xi}(t),\dot{\gamma}
_{x,\xi}(t),\dot{\gamma}
_{x,\xi}(t)) \, \dv\,dt =0,
\end{align}
for any $\xi\in \s^{n-1}$ and any $\rho_j\in C_0^{\infty}(\widetilde{\M}\setminus \M),\;j=1,2$.
\subsection{ End of the proof of Theorem \ref{th2}}
Let $x_0\in \widetilde{\M} \setminus \M$, $\xi \in \s^{n-1}$ and consider a positive function $\rho_0\in C_0^{\infty}([0,1])$ such that $\| \rho_0\|_{L^1([0,1])}=1$. For $h>0$, we define $$\rho_h(x)=h^{-n}\rho_0\big(\frac{d_{\g}(x,x_0)}{h}\big).$$
Since $\mathrm{supp}(\rho_h) \subset B(x_0,h)$ then for $h$ sufficiently small, we have  $\mathrm{supp}(\rho_h) \cap \M =\emptyset$ and $\forall x\in \mathrm{supp}(\rho_h)$ and $\gamma_{x,\xi}(\pm T) \notin \M$. Take $\rho_1=\rho_h$, then  as $h \rightarrow 0$, the identity (\ref{4.9}) gives
\begin{equation}\label{4.10}
  \int_0^T\!\!\!\rho_2(x_0)k(\gamma_{x_0,\xi}(t),\dot{\gamma}
_{x_0,\xi}(t),\dot{\gamma}
_{x_0,\xi}(t)) \,dt=0. \end{equation}
Now let us choose the function $\rho_2$ to be a good approximation of the sign of $k$. Let $$U :=\{x \in \widetilde{\M} \setminus \M,\; k(\gamma_{x,\xi}(t),\dot{\gamma}
_{x,\xi}(t),\dot{\gamma}
_{x,\xi}(t)) > 0\}.$$
From the continuity of $k$, it follows that $U$ is an open subset of $\widetilde{\M} \setminus \M$ . Let $(K_m)_{m\in \N}$ be a sequence of compact sets such that $\displaystyle \bigcup_m K_m= U$ and $K_m \subset  K_{m+1}$. For $m \in \N$, let $\rho_m\in C_0^{\infty }(\widetilde{\M} \setminus \M;\R)$ such that $\chi_{K_m} \leq \chi_m \leq \chi_U$, and let
$$\rho_m = 2\chi_m-1.$$ Thus, using the continuity properties of $k$ again, we obtain $$\displaystyle \lim_{m \rightarrow +\infty}
\rho_m(x_0)k(\gamma_{x_0,\xi}(t),\dot{\gamma}
_{x_0,\xi}(t),\dot{\gamma}
_{x_0,\xi}(t)) = |k(\gamma_{x_0,\xi}(t),\dot{\gamma}
_{x_0,\xi}(t),\dot{\gamma}
_{x_0,\xi}(t))|.$$
By choosing $\rho_2=\rho_m$ and taking the limit as $m \rightarrow +\infty$, (\ref{4.10}) gives that
  \begin{equation*}
 \int_{0}^T\!\!\!|k(\gamma_{x_0,\xi}(t),\dot{\gamma}
_{x_0,\xi}(t),\dot{\gamma}
_{x_0,\xi}(t))| \,dt =0.
\end{equation*}
We proceed like in the proof of the estimate (\ref{ineq -T T}), to obtain \begin{equation*}  \int_{-T}^T\!\!\!|k(\gamma_{x_0,\xi}(t),\dot{\gamma}
_{x_0,\xi}(t),\dot{\gamma}
_{x_0,\xi}(t))| \,dt=0.
\end{equation*}
  Recall that we have assumed that $k_j(x,\theta,\theta')=\sigma_j(x)\chi(\theta,\theta')$, $j=1,2$, with $\sigma_j\in C^{1}(\M)$, $\chi\in C(S_x(\M) \times S_x(\M))$ for a.e. $x \in \M$ and $\chi(\theta,\theta) \neq 0, \forall \theta \in S_x\M$. We set $$\sigma=\sigma_1-\sigma_2.$$ We extend $\sigma$ by $0$ outside $\M$. We can also extend $\chi$ to be continuous on $S_x(\widetilde{\M}) \times S_x(\widetilde{\M}))$ such that $\chi(\theta,\theta)\neq 0,\, \forall \theta \in S_x\widetilde{\M}$. % \begin{equation*}
%\int_0^T\!\!\!\int_{\widetilde{\M} \setminus \M}\!\!\! \rho_1^2(x)\sigma(\gamma_{x,\xi}(t)) \chi(\dot{\gamma}
%_{x,\xi}(t),\dot{\gamma}
%_{x,\xi}(t)) \, \dv\,dt=0,\; \forall \xi \in \s^{n-1}.
%\end{equation*}
Then there exist a constant $m_\chi >0$ such that for a.e $x_0\in \widetilde{\M } \setminus \M$ we have $ |\chi(\dot{\gamma}
_{x_0,\xi}(t),\dot{\gamma}
_{x_0,\xi}(t))| \geq m_\chi, \; \forall t \in [-T,T]$. Thus we have
 \begin{equation*}
  \int_{-T}^T\!\!\!|\sigma(\gamma_{x_0,\xi}(t)) |\,dt =0,
\end{equation*} and then \begin{equation*}
  \int_{-T}^T\!\!\!\sigma(\gamma_{x_0,\xi}(t)) \,dt =0,
\end{equation*}
%Now, take $(x,\theta)\in S\M$ and let $t_0\in \R$ such that $\gamma_{x,\xi}(t_0)\in \widetilde{\M} \setminus \M$. We set
%$x_0=\gamma_{x,\xi}(t_0)$ and $\xi=\dot{\gamma}_{x,\theta}(t_0)$.

 Keeping in mind that $\mathrm{supp}(\sigma) \subset \M$ and that $T> Diam_{\g} \M$ and proceeding like we did in the proof of Lemma (\ref{lem3.4}), we get
  \begin{equation*}
  \X(\sigma)(x',\theta') =0,\;\forall (x',\theta')\in \p_-S\M.
\end{equation*} and then $\X(\sigma)=0$. The stability estimate (\ref{ray inverse}) and the regularity of $\sigma$  yield to $\sigma=0$ on $\M$.
This completes the proof of Theorem \ref{th2}.
%%%%%%%%%%%%%%%%%%%%%%%%%%%%%%%%%%%%%%%%%%%%%%%%%%%%
%%%%%%%%%%%%%%%%%%%%%%%%%%%%%%%%%%%%%%%%%%%%%%%%%%%
\appendix
\section{}
\setcounter{equation}{0}
%\subsection{}
\begin{lemma}\label{lemma}
 Assume that $\M$ is a simple Riemannian manifold.  Let $T>0$, $p\geq 1$, $a\in \mathscr{A}_0(C_0)$ and $k\in \mathscr{K}$. Suppose that $f_-\in \mathcal{L}^p(\Sigma_-)$. If $q\in L^p(Q)$ and $u_0\in L^p(S\M)$ then the problem
 \begin{equation}\label{ibvpgeneral}
 \left\{
\begin{array}{lll}
\partial_t u(t,x,\theta)+\D u(t,x,\theta)+a(x,\theta)u(t,x,\theta)=\I_k[u](t,x,\theta) +q(t,x,\theta)  & \textrm{in }\; Q_T,\cr
u(0, \cdot,\cdot )=u_0 & \textrm{in }\; S\M ,\cr
u=f_- & \textrm{on } \; \Sigma_-.
\end{array}
\right.
 \end{equation}
 admits a unique weak solution $u \in C^0([0,T];\mathcal{W}^p)$.\\
 If we assume that $u_0\in \mathcal{W}^p$ and $q\in C^1([0,T];L^p(S\M))$ then $u$ is a strong solution and $$u\in  C^1([0,T];L^p(S\M))\cap  C^0([0,T];\mathcal{W}^p).$$
 Moreover, there exist a constant $C>0$ such that
\begin{equation}\label{estimate}
  \norm{u(t,\cdot,\cdot)}_{L^p(S\M)}+ \norm{u}_{\mathcal{L}^p(\Sigma_+)} \leq  C \para{ \norm{u_0}_{L^p(S\M)}+ \norm{f_-}_{\mathcal{L}^p(\Sigma_-)}+ \norm{q}_{L^p(Q)}},\;\forall t\in [0,T].
\end{equation}
\end{lemma}
\begin{proof}
To prove that the system (\ref{ibvpgeneral}) is uniquely solvable, we can mime the proof, detailed in the Euclidian case, in \cite{[DL]} (Theorem $3$ p.229). We begin by  proving  that the operator $T$ is an infinitesimal generator of a $C^0$-semigroup on the Banach space $\mathcal{W}^p$ and then  we infer that the initial boundary value problem (\ref{ibvpgeneral}) admits a unique weak solution. With the additional hypothesis on $u_0$ and $q$, we deduce that the weak solution is a strong (classical) one.\\
Let us prove the estimate (\ref{estimate}). Multiplying the first equation of (\ref{ibvpgeneral}) by $\abs{u}^{p-2} \overline{u}$, integrating over $S\M$ then taking its real part, we get
\begin{multline}\label{integrer}
  \int_{S\M} \frac{1}{p} \p_t \abs{u}^{p}(t) \dvv (x,\theta)+ Re \int_{S\M} \abs{u}^{p-2}(t) \overline{u}(t)\D u(t) \dvv (x,\theta)+\int_{S\M} \abs{u}^{p}(t) a  \dvv (x,\theta)\cr
  =Re \int_{S\M} \abs{u}^{p-2}(t) \overline{u}(t)\I_k[u](t) \dvv (x,\theta)+ Re \int_{S\M} \abs{u}^{p-2}(t) \overline{u}(t)q(t) \dvv (x,\theta)
\end{multline}
Combining  ( \ref{div theta}), (\ref{nabla h u}) and (\ref{Du nabla u}), we obtain $ \overset{\hh}{\dive}(\theta \abs{u}^p)=\overset{\hh}{\nabla}_i(\theta^i \abs{u}^p)=\theta^i\overset{\hh}{\nabla}_i(\abs{u}^p)=p\abs{u}^{p-2} Re(\overline{u}\D u)$. Applying Gauss-Ostrogradskii formula for the horizontal  divergence (\ref{divergence formula horizontal}), we get  $$ Re \int_{S\M} \abs{u}^{p-2}(t) \overline{u}(t)\D u(t) \dvv (x,\theta)= \frac{1}{p} \int_{\p S\M}  \abs{u}^p \langle \theta, \nu(x) \rangle \, \dss.$$
The equality (\ref{integrer}) becomes
\begin{multline*}
  \int_{S\M} \p_t \abs{u}^{p}(t) \dvv (x,\theta)+ \int_{\p_+S\M} \abs{u}^{p}(t)  \langle \theta, \nu(x) \rangle \, \dss= \cr \int_{\p_-S\M} \abs{u}^{p}(t)  \langle \theta, \nu(x) \rangle \, \dss -p \int_{S\M} \abs{u}^{p}(t) a(x,\theta)  \dvv (x,\theta)
  \cr +p Re \int_{S\M} \abs{u}^{p-2}(t) \overline{u}(t)\I_k[u](t) \dvv (x,\theta) +p Re \int_{S\M} \abs{u}^{p-2}(t) \overline{u}(t)q(t) \dvv (x,\theta).
\end{multline*}
Let $p'$ be the integer that satisfies $ \frac{1}{p}+ \frac{1}{p'}=1$. By H\"older's inequality, we have
$$\abs{Re \int_{S\M} \abs{u}^{p-2}(t) \overline{u}(t)\I_k[u](t) \dvv (x,\theta)} \leq C_1^{\frac{1}{p'}} C_2^{\frac{1}{p}} \norm{u}_{L^p(S\M)}^p$$ and by the Young inequality, we get $$\abs{ Re \int_{S\M} \abs{u}^{p-2}(t) \overline{u}(t)q(t) \dvv (x,\theta)}\leq {\frac{1}{p}} \norm{q}_{L^p(S\M)}^p +{\frac{1}{p'}} \norm{u}_{L^p(S\M)}^p.$$
We infer that
 \begin{multline}\label{I}
\p_t  \norm{u}_{L^p(S\M)}^p + \int_{\p_+S\M} \abs{u}^{p}(t)  \langle \theta, \nu(x) \rangle \, \dss \leq \int_{\p_-S\M} \abs{u}^{p}(t)  \abs{\langle \theta, \nu(x) \rangle }\, \dss \cr
 +\norm{a}_{L^{\infty}(S\M)} \norm{u}_{L^p(S\M)}^p+p C_1^{\frac{1}{p'}} C_2^{\frac{1}{p}} \norm{u}_{L^p(S\M)}^p+
\norm{q}_{L^p(S\M)}^p+{\frac{p}{p'}} \norm{u}_{L^p(S\M)}^p.
\end{multline}
Hence, there exist a constant $C$ depending on $p,\,C_0, \,C_1,\,C_2$ such that
\begin{equation}\label{II}
 \p_t  \norm{u}_{L^p(S\M)}^p \leq  \int_{\p_-S\M} \abs{u}^{p}(t)  \abs{\langle \theta, \nu(x) \rangle }\, \dss + \norm{q}_{L^p(S\M)}^p+C \norm{u}_{L^p(S\M)}^p.
\end{equation}
Integrating over $[0,t]$,we get
\begin{equation*}
   \norm{u(t)}_{L^p(S\M)}^p \leq \norm{f_-}^p_{\mathcal{L}^p(\Sigma_-)}+ \norm{q}_{L^p(Q)}^p+  \norm{u(0)}_{L^p(S\M)}^p+ C \displaystyle \int_{0}^{t} \norm{u(s)}_{L^p(S\M)}^p  .
\end{equation*}
By Gr\"onwall's inequality, there exist a constant $C$ such that
\begin{equation}\label{I1}
   \norm{u(t)}_{L^p(S\M)}^p \leq C \big (\norm{f_-}^p_{\mathcal{L}^p(\Sigma_-)}+ \norm{q}_{L^p(Q)}^p+   \norm{u(0)}_{L^p(S\M)}^p\big ) .
\end{equation}
Now integrating (\ref{I}) over $[0,T]$, we arrive to
\begin{equation*}
   \norm{u(T)}_{L^p(S\M)}^p+\norm{u}^p_{\mathcal{L}^p(\Sigma_+)} \leq \norm{f_-}^p_{\mathcal{L}^p(\Sigma_-)}+ \norm{q}_{L^p(Q)}^p+  \norm{u(0)}_{L^p(S\M)}^p+ C \displaystyle \int_{0}^{T} \norm{u(s)}_{L^p(S\M)}^p  .
\end{equation*}
In light with (\ref{I1}), this implies that \begin{equation}\label{I2}
\norm{u}^p_{\mathcal{L}^p(\Sigma_+)} \leq C \big ( \norm{f_-}^p_{\mathcal{L}^p(\Sigma_-)}+ \norm{q}_{L^p(Q)}^p+  \norm{u(0)}_{L^p(S\M)}^p\big )  .
\end{equation}
Combining (\ref{I1}) and  (\ref{I2}), we obtain the claimed estimate.
\end{proof}
%\subsection{}
\begin{lemma}\label{lemA2}
Let $(\lambda_n)_{n\in \N}$ be a real sequence going towards $+ \infty$, then the sequence $e^{i \lambda_n \tau_-(x,\theta)}$ converges  towards zero for the weak topology of $L^2(S\M)$.
\end{lemma}
\begin{proof}
  Let $f \in L^2(S\M)$, we aim to prove $\langle f,e^{i \lambda_n \tau_-} \rangle \longrightarrow 0$ as $n \rightarrow + \infty$.
  Using the property \begin{equation*}
  \tau_-(\gamma_{x,\theta}(s),\dot{\gamma}_{x,\theta}(s))=
  \tau_-(x,\theta)-s \; \mbox{ and }\; \tau_-=0 \; \mbox{ on } \;\p_-S\M,
   \end{equation*}
  we have
  \begin{multline*}
 \displaystyle \int_{S\M} f(x,\theta) e^{-i \lambda_n \tau_-(x,\theta)} \dvv (x,\theta)= \displaystyle  \int_{\p_-S\M}  \int_0^{\tau_+(x',\theta')} f(\gamma_{x',\theta'}(s),\dot{\gamma}_{x',\theta'}(s))e^{i \lambda_n s} ~ds \mu(x',\theta') \dss
  \cr =\displaystyle  \int_{\p_-S\M}  \int_0^{\tau_+(x',\theta')}\alpha_{x',\theta'}(s)e^{i \lambda_n s} ~ds \mu(x',\theta') \dss,
\end{multline*}
where we set $\alpha_{x',\theta'}(s)= f(\gamma_{x',\theta'}(s),\dot{\gamma}_{x',\theta'}(s))$. \\
We fix $(x',\theta')\in \p_-S\M$ and we consider \begin{equation}\label{famille totale}\mathfrak{F}=\Big\{\chi_{[a,b]},\;[a,b]\subset [0,\tau_+(x',\theta')]\Big\}
\end{equation} which is a complete set in $L^2([0,\tau_+(x',\theta')])$. We define $(l_n)_n$ a sequence of linear forms on $L^2([0,\tau_+(x',\theta')])$ by setting  $$l_n(\alpha)= \int_0^{\tau_+(x',\theta')}\alpha(s)e^{-i \lambda_n s} ~ds .$$
We have $\abs{l_n(\alpha)} \leq \tau_+(x',\theta')\|\alpha \|_{L^2([0,\tau_+(x',\theta')])}$. Then the norm of $l_n$ as a linear bounded operator satisfies $\|l_n \|\leq \tau_+(x',\theta')$. Hence, we have
\begin{equation}\label{norm ln} \displaystyle \sup_{n \geq 1}\|l_n \|\leq \tau_+(x',\theta') < \infty.
\end{equation}
As we have obviously $\abs{l_n(\chi_{[a,b]})} \leq\frac{2}{\abs{\lambda_n}} \rightarrow 0$ as $n \rightarrow \infty$ which gives that $l_n(\chi_{[a,b]})\rightarrow 0= l(\chi_{[a,b]})$, then combining with (\ref{famille totale}) and (\ref{norm ln}), we obtain that
\begin{equation*}
  \forall\alpha \in L^2([0,\tau_+(x',\theta')]), \;\; l_n(\alpha) \rightarrow l(\alpha)=0  \;\mbox{ as } \; n \rightarrow \infty.
\end{equation*}
Therefore the sequence $(l_n(\alpha))_n$ is bounded independently of $n$. Applying the Lebesgue's Theorem and replacing $\alpha$ by $\alpha_{x',\theta'}$, we end up by
  \begin{multline*}
\displaystyle \lim_{n \rightarrow +\infty} \displaystyle \int_{S\M} f(x,\theta)e^{i \lambda_n \tau_-(x,\theta)}  \dvv (x,\theta)= \cr \displaystyle  \int_{\p_-S\M}  \displaystyle \lim_{n \rightarrow +\infty} \int_0^{\tau_+(x',\theta')} \alpha_{x',\theta'}(s)e^{i \lambda_n s} ~ds \mu(x,\theta) \dss =0.
\end{multline*}
This completes the proof of the Lemma.
\end{proof}
\subsection*{Proof of Lemma \ref{properties psi h}}\label{proof psi}
The property (\ref{prop1}) is obvious. Let us prove property (\ref{prop2}).
The Poisson kernel of $B(0,1)\subset T_x \widetilde{\M}$ is defined as follows:
$$
P(\theta,\xi)=\frac{1-\abs{\theta}^2}{\alpha_n\abs{\theta-\xi}^n},\quad \theta\in B(0,1);\,\, \xi\in S_x\widetilde{\M}.
$$
For $0<h<1$, we define $\Psi_h: S_x\widetilde{\M} \times S_x\widetilde{\M} \to \R$ as
$$
\Psi_h(\theta,\xi)=P(h\theta,\xi).
$$
Let $P_0$ be the Poisson kernel for the Euclidian unit ball $B_0(0,1)\subset\R^n$ i.e.,
$$
P_0(\hat{\theta},\hat{\xi})=\frac{1-\abs{\hat{\theta}}_0^2}{\alpha_n\abs{\hat{\theta}-\hat{\xi}}_0^n},\quad \hat{\theta}\in B_0(0,1);\,\, \hat{\xi}\in \mathbb{S}^{n-1}.
$$
where $\abs{\cdot}_0$ is the Euclidian norm of $\R^n$. From the well known properties of $P_0$, we have
$$
\int_{\mathbb{S}^{n-1}}P_0(\kappa\hat{\theta},\hat{\xi}) d\omega_0(\hat{\xi})=1, \quad\textrm{for all}\, \kappa\in (0,1),\,\, \hat{\theta}\in \mathbb{S}^{n-1}.
$$
Let $\gamma=\sqrt{\g}:=(\gamma_{ij})$ be definite symetric positive matrix such that $\gamma^2=(g_{ij})$, then we get
$$
P(\theta,\xi)=\frac{1-\abs{\gamma^{-1}\theta}_0^2}{\alpha_n\abs{\gamma^{-1}\theta-\gamma^{-1}\xi}_0^n},\quad \theta,\xi\in S_x\widetilde{\M}.
$$
We deduce from the change of variable $\hat{\theta}=\gamma^{-1}\theta$ that
\begin{eqnarray*}
\int_{S_x\widetilde{\M}}P(h\xi,\theta) d\omega_x(\theta)=\int_{S_x\widetilde{\M}}P_0(h\gamma^{-1}\xi,\gamma^{-1}
\theta) =\frac{1}{\det\gamma}  \int_{\mathbb{S}^{n-1}}P_0(h\hat{\xi},\hat{\theta})(\det\gamma) d\omega_0(\hat{\theta})=1.
\end{eqnarray*}
This completes the proof of (\ref{prop2}).
Let now recall that for any $\hat{v}\in C(\mathbb{S}^{n-1})$, we have the following Poisson's formula:
\begin{equation*}
  \displaystyle \lim_{h \rightarrow 1} \int_{\mathbb{S}^{n-1}}P_0(h\hat{\xi},\hat{\theta})\hat{v}(\hat{\theta}) d\omega_0(\hat{\theta})=v(\hat{\theta}),\; \forall \hat{\xi}\in \mathbb{S}^{n-1}.
  \end{equation*}
  Then, for any $v\in C(S_x\widetilde{\M})$, we have
  \begin{equation*}
   \displaystyle \lim_{h \rightarrow 1}
 \int_{S_x\widetilde{\M}}P(h\xi,\theta)v(\theta) d\omega_x(\theta)= \displaystyle \lim_{h \rightarrow 1} \int_{\mathbb{S}^{n-1}}P_0(h\hat \xi,\hat{\theta})
\hat{v}(\hat{\theta}) d\omega_0(\hat{\theta}),
\end{equation*} where $\hat{\xi}=\gamma^{-1}\xi$ and $\hat{\theta}=\gamma^{-1} \theta$ and $\hat{v}(\hat{\theta})=v(\gamma \hat{\theta})$.
Hence \begin{equation*}
   \displaystyle \lim_{h \rightarrow 1}
 \int_{S_x\widetilde{\M}}P(h\xi,\theta)v(\theta) d\omega_x(\theta)= \hat{v}( \hat{\xi} )=v(\xi).
\end{equation*}
This completes the proof of Lemma \ref{properties psi h}.
\subsection*{Acknowledgements}
 The author would like to thank Pr. Mourad Bellassoued for his enriching discussions and many helpful suggestions.
 %%%%%%%%%%%%%%%%%
%%%%%%%%%%%%%%%%%%%%%%%%%%%%%%%%%%%%%%%%%%%%%%%%%%%%%%%%%%%%%%%%%%%%%%%%%%%%%%%%%%%%%%%%%%%%

\end{document}